\documentclass[11pt]{amsart}
\usepackage{amsmath}
\usepackage{bbding}
\usepackage{tikz}
\usepackage{amsmath,amssymb}
\usepackage{mathrsfs}
\theoremstyle{plain}
\newtheorem{theorem}{Theorem}[section]
\newtheorem{lemma}[theorem]{Lemma}
\newtheorem{prop}[theorem]{Proposition}
\newtheorem{corollary}[theorem]{Corollary}

\newtheorem{assumption}[theorem]{Assumption}

\theoremstyle{definition}
\newtheorem{remark}[theorem]{Remark}
\newtheorem{definition}[theorem]{Definition}

\numberwithin{equation}{section}

\def\be{\begin{equation}}
\def\ee{\end{equation}}

\begin{document}

\title[Asymptotic behavior of complete conformal metric]
{\textbf{Asymptotic behavior of complete conformal metric near singular boundary}}
\author[Shen]{Weiming Shen}
\address{School of Mathematical Sciences\\
Capital Normal University\\
Beijing, 100048, China}
\email{wmshen@aliyun.com}
\author[Wang]{Yue Wang}
\address{School of Mathematical Sciences\\
Capital Normal University\\
Beijing, 100048, China}
\email{yuewang37@aliyun.com}

\begin{abstract}
 The boundary behavior of the singular Yamabe problem has been extensively studied near sufficiently smooth boundaries, while less is known about the asymptotic behavior of solutions near singular boundaries. In this paper, we study the asymptotic behaviors of solutions to the singular Yamabe problem with negative constant scalar curvature near singular boundaries and derive the optimal estimates for the background metric which is not necessarily conformally flat. In particular, we prove that the solutions are well approximated by the solutions in tangent cones at singular points on the boundaries.
\end{abstract}

\thanks{The first author acknowledges the support of NSFC Grant 12371208.
The second author acknowledges the support of NSFC Grant 12371236.}
\maketitle
\section{Introduction}\label{sec-Intro}

Let $(M,g)$ be a smooth Riemannian manifold of dimension $n$, for some $n\ge 3$, which is either compact
without boundary or noncompact and complete. Assume $\Omega\subset M$ is a smooth domain.
We consider the following problem:
\begin{align}
\label{eq-MEq1} \Delta_{g} u -\frac{n-2}{4(n-1)}S_gu&= \frac14n(n-2) u^{\frac{n+2}{n-2}} \quad\text{in }\,\Omega,\\
\label{eq-MBoundary2}u&=+\infty\quad\text{on }\partial \Omega ,
\end{align}
where $S_g$ is the scalar curvature of $(M,g)$. Geometrically, if $u$ is a positive solution of \eqref{eq-MEq1}-\eqref{eq-MBoundary2},
the singular Yamabe metric $u^{\frac{4}{n-2}}g$ is a complete metric with a constant scalar curvature $-n(n-1)$ on $\Omega$.
According to Loewner and Nirenberg \cite{Loewner&Nirenberg1974} for $(M,g)=(S^n, g_{S^n})$ and
Aviles and McOwen \cite{AM1988DUKE} for the general case,
\eqref{eq-MEq1}-\eqref{eq-MBoundary2} admits a unique positive solution when $\Omega$ is a bounded smooth domain with an $(n-1)$-dimensional boundary.

If $(M,g)=(\mathbb R^n, g_E)$,
then \eqref{eq-MEq1}-\eqref{eq-MBoundary2} reduces to
\begin{align}
\label{eq-MainEq} \Delta u  &= \frac14n(n-2) u^{\frac{n+2}{n-2}} \quad\text{in }\,\Omega,\\
\label{eq-MainBoundary}u&=+\infty\quad\text{on }\partial \Omega.
\end{align}
When $\Omega=\mathbb R^n_{+}:=\{x \in \mathbb R^n|x_n>0\}$, the unique positive solution of \eqref{eq-MainEq}-\eqref{eq-MainBoundary}
is given by $u=x_n^{-\frac{n-2}{2}}$. In this case, $x_n$ equals the distance of $x$ to $\partial \Omega$ with respect to the Euclidean metric.

The study of boundary behaviors of \eqref{eq-MEq1}-\eqref{eq-MBoundary2} and \eqref{eq-MainEq}-\eqref{eq-MainBoundary} has a rich history.
To name a few,
Loewner and Nirenberg \cite{Loewner&Nirenberg1974} first studied asymptotic behaviors
of solutions of \eqref{eq-MainEq}-\eqref{eq-MainBoundary} and proved
an estimate involving leading terms.
For bounded $C^2$ domains $\Omega$, let $d$ be the distance function to $\partial\Omega$
and $u$ be a solution of \eqref{eq-MainEq}-\eqref{eq-MainBoundary}.
Loewner and Nirenberg \cite{Loewner&Nirenberg1974} proved, for $d$ sufficiently small,
\begin{equation}\label{eq-EstimateDegree1}|d^{\frac{n-2}{2}}u-1|\le Cd, \end{equation}
where $C$ is a positive constant depending only on certain geometric quantities of $\partial\Omega$.
When $\Omega$ has a smooth boundary, an
estimate of solutions of \eqref{eq-MEq1}-\eqref{eq-MBoundary2} up to an arbitrarily finite order was established by
Andersson, Chru\'sciel and Friedrich \cite{ACF1982CMP} and Mazzeo \cite{Mazzeo1991}. In particular,
\begin{equation*}u=d^{-\frac{n-2}{2}}\Big[1+\frac{n-2}{4(n-1)}H_{\partial\Omega}d+O(d^2)\Big],\end{equation*}
where $d$ is the distance to $\partial\Omega$ and $H_{\partial\Omega}$ is the mean curvature of $\partial\Omega$
with respect to the interior
unit normal vector of $\partial\Omega$. Hence, the estimate
\eqref{eq-EstimateDegree1} still holds for solutions of \eqref{eq-MEq1}-\eqref{eq-MBoundary2} when the distance function to $\partial\Omega$ with respect to the metric $g$ is sufficiently small.

The singular Yamabe problem has many geometric applications in recent years, see \cite{CMY2022}, \cite{ChenLaiWang2019}, \cite{Graham2017}, \cite{GrahamGursky2021} and \cite{GurskyHan2017}.  The boundary behaviors of solutions to \eqref{eq-MEq1}-\eqref{eq-MBoundary2}  play an important role in these works.

 Next, we review the results on the boundaries with singularities. Compared to the case where the boundary is sufficiently smooth, there are few results on the asymptotic behavior of solutions in domains with singular boundaries.
 In \cite{hanshen2}, Han and the first author studied the asymptotic behaviors of solutions of \eqref{eq-MainEq}-\eqref{eq-MainBoundary}
in singular domains, and proved that these solutions are well approximated by the
corresponding solutions in tangent cones near singular points on the boundaries.
In \cite{HanJiangShen}, Han, Jiang and the first author studied asymptotic behaviors of solutions to \eqref{eq-MainEq}-\eqref{eq-MainBoundary}
in finite cones and established the optimal
asymptotic expansions in terms of the corresponding solutions in infinite cones.
All these
results are established in domains in $(\mathbb R^n, g_E)$.

In this paper, we will study the
asymptotic behaviors of local positive solutions of \eqref{eq-MEq1}-\eqref{eq-MBoundary2}
near singular points on $\partial\Omega$, when the background metric is not necessarily conformally flat.
For convenience, we formulate the domains and the background metrics in the following assumption.
\begin{assumption}\label{assumption-basic} Let $g$ be a Riemannian metric on $B_{2}(0)\subseteq\mathbb R^n$ where $n\geq 3$,
 and $(x_1,\cdots,x_n)$ be a normal coordinate system of $g$ with $0$ as the origin. Let $\Omega \subseteq B_{2}(0) $ be a domain with
$0\in \partial \Omega$ and  $\partial\Omega \bigcap \overline{B_{1}(0)} $ being Lipschitz.
\end{assumption}

 Now we consider a local positive solution of \eqref{eq-MEq1}-\eqref{eq-MBoundary2}, namely, $u>0$ in $\Omega\bigcap B_{1}(0)$ satisfies
\begin{align}
\label{eq-MEq} \Delta_{g} u -\frac{n-2}{4(n-1)}S_gu&= \frac14n(n-2) u^{\frac{n+2}{n-2}} \quad\text{in }\,\Omega\bigcap B_{1}(0),\\
\label{eq-MBoundary}u&=+\infty\quad\text{on }\partial \Omega \bigcap B_{1}(0),
\end{align}
where $S_g$ is the scalar curvature of $g$. Our main interest in this paper will be to describe the asymptotic behavior of such a local positive solution near singular boundary.

Our main result in this paper is given by the following theorem.

\begin{theorem}\label{main reslut}
Let $\Omega$ and $g$ be as in Assumption \ref{assumption-basic},
 and, for some integer $k\le n$,  let $\partial\Omega$ in $B_1(0)$ consist of
$k$ $C^{2}$-hypersurfaces $S_1, \cdots, S_k$ intersecting at $0$ with the property that the normal vectors
of $S_1, \cdots, S_k$ at $0$ are linearly independent. Suppose $ u \in
C^{\infty}(\Omega\bigcap B_{1}(0))$ is a positive solution of \eqref{eq-MEq}-\eqref{eq-MBoundary},
and $u_{V_{0}}$ is the solution of \eqref{eq-MainEq}-\eqref{eq-MainBoundary} in the tangent cone $V_{0}$
of $\Omega$ at $0$.
Then, there exists a constant $r\leq \frac{1}{2}$, and a 
$C^{2}$-diffeomorphism $T_{g}$:
$B_r (0)\rightarrow T_g(B_r(0))\subseteq\mathbb R^n$, with
$T_g(\Omega \bigcap B_r (0))= V_{0}\bigcap T_g(B_r(0))$ and
$T_g(\partial\Omega \bigcap B_r (0))= \partial V_{0}\bigcap T_g(B_r(0))$, such that,
for any $x\in \Omega \bigcap B_{r/2}(0)$,
\begin{align}
\label{main-estimate} \Big|\frac{u(x)}{u_{V_{0}}(T_gx)}-1\Big|\leq Cd_{g}(x,0),
\end{align}
where $d_{g}(x,0)$ is the distance from $x$ to $0$  with respect to the metric $g$ and $C$ is a positive constant depending only on $n$, $g$
and the geometry of $\partial\Omega$ near $0$.
\end{theorem}

The map $T_g$ and concept of tangent cones are introduced in Section \ref{sec-pre}. $T_g$ is determined
by $d_{g,i}$, $i=1,...,k$, where $d_{g,i}$ is the signed distance of $x$ to $S_i$  with respect to the metric $g$.
In fact, $u_{V_{0}}(T_gx)$ can be expressed by a function $f_{V_{0}}(d_{g,1}(x),...,d_{g,k}(x))$ which does not depend on the choice of coordinate system and $f_{V_{0}}$ depends only on $V_0$.
When $k=1$, $V_{0}$ is conjugate to $\mathbb R^n_{+}$ and  $f_{V_{0}}(d_{g,1}(x))=d_{g,1}^{-\frac{n-2}{2}}(x)$.
The constant $C$ depends on the geometry of $\partial\Omega$ and $g$ which will be specified in the proof of Theorem \ref{main reslut}.
The estimate \eqref{main-estimate} generalizes \eqref{eq-EstimateDegree1} and the power one of the
distance on the right-hand side cannot be improved in general if for some $i,$ $S_i$ is curved in general.

Geometrically, this theorem asserts that we can approximate the original singular Yamabe metric near the singularity with a corresponding singular Yamabe metric defined in the tangent cone obtained by blowing up at the singularity. Compared with the original metric, the singular Yamabe metric in the tangent cone has more symmetry and is simpler in structure.

In particular, when $\partial \Omega$ near the origin is a cone with its vertex at the origin,  we can improve the estimate \eqref{main-estimate}.

\begin{theorem}\label{main reslut2}
Let $\Omega$ and $g$ be as in Assumption \ref{assumption-basic},
 and $V_0$ be a Lipschitz infinite cone with its vertex at the origin. Suppose $\Omega\bigcap B_{1}(0)=V_0\bigcap B_{1}(0)$,  $u\in C^{\infty}(\Omega\bigcap B_{1}(0))$ is a positive solution of \eqref{eq-MEq}-\eqref{eq-MBoundary},
and $u_{V_{0}}$ is the solution of \eqref{eq-MainEq}-\eqref{eq-MainBoundary} in the cone $V_{0}$.
Then there exists a constant $\alpha>1$ and a positive constant $r$, such that,
for any $x\in  \Omega \bigcap B_{r/2}(0)$,
\begin{align}
\label{main-estimate2} \Big|\frac{u(x)}{u_{V_{0}}(x)}-1\Big|\leq Cd_{g}^{\alpha}(x,0),
\end{align}
where
 $C$ is a positive constant depending only on $n$, $g$
and the geometry of $\partial\Omega$ near $0$.
\end{theorem}

The index $\alpha$ can be $2$ in Theorem \ref{main reslut2} when $n\geq6$. Also, $\alpha$ in Theorem \ref{main reslut2} can always be given by $\alpha=2$ when $V_0$ is convex, see Corollary \ref{345-convex cone}.

The main challenge is to derive the sharp estimates in Theorem \ref{main reslut} and Theorem \ref{main reslut2}.
 For example, if it is $Cd_{g}^{\alpha}(x,0)$ for some positive constant $\alpha$ sufficiently small   instead of $Cd_{g}(x,0)$ on the right hand side of \eqref{main-estimate}, then the proof can be greatly simplified and  a routine application of barrier techniques will be enough.  However,  deriving the sharp estimates in \eqref{main-estimate} is much more complicated and delicate.

We now describe the proof of Theorem \ref{main reslut}. Basically,
we view the equation \eqref{eq-MEq} as a perturbation of the equation \eqref{eq-MainEq}. Then we use the maximum principle to prove a local solution
of \eqref{eq-MEq1}-\eqref{eq-MBoundary2} is close to a local solution of \eqref{eq-MainEq}-\eqref{eq-MainBoundary} in the corresponding region.
If $\partial \Omega$ is at least $C^2$, then the distance function $d(x,\partial\Omega)$ is at least $C^2$ near $\partial\Omega$ when $x$ is sufficient close to $\partial\Omega$,
and hence can be used to construct
various barrier functions.
When $\partial \Omega$ is merely Lipschitz,
the distance function is not $C^2$ near $\partial\Omega$
anymore and cannot be used to construct barrier functions.
We will overcome this difficulty by carefully solving certain equations, then using these solutions to construct
various barrier functions. Especially, due to the growth rate of solutions of \eqref{eq-MainEq}-\eqref{eq-MainBoundary}  near boundaries in low dimensions, we need to derive certain a priori estimates. For this purpose, we  study the eigenvalue problem of a specific
elliptic operator on spherical domains whose coefficients are singular on the boundary and estimate the lower bound of the first eigenvalue. The first eigenvalue and eigenfunction play an essential role in the derivation of
the optimal estimates.
Then, we compare a solution of \eqref{eq-MainEq}-\eqref{eq-MainBoundary} in the curved domain with
corresponding solutions in the  tangent cones. The properties of the $C^{2}$-diffeomorphism $T_{g}$ will be used in such comparison.

We now compare results in this paper with corresponding results for the singular Yamabe equations
corresponding to positive scalar curvatures. First consider the positive solutions of the Yamabe equation of the form
\begin{equation}\label{eq-Yamabe}
\Delta u=-\frac14n(n-2)u^{\frac{n+2}{n-2}}, \quad\text{in }B_1(0)\setminus \{0\}.\end{equation}
Let $u$ be a positive solution of \eqref{eq-Yamabe} in $B_1(0)\setminus \{0\}$,
with a nonremovable singularity at the origin.
Based on the Alexandrov reflection method, Caffarelli, Gidas, and Spruck \cite{CaffarelliGS1989}
first proved that $u$ is asymptotic to a radial singular solution of \eqref{eq-Yamabe} in $\mathbb R^n\setminus\{0\}$.
Subsequently in \cite{KorevaarMPS1999}, Korevaar, Mazzeo, Pacard and Schoen
gave a simplified proof and expanded solutions to the next order. Rencently, Han, Li and Li \cite{Han-Li-Li} established an expansion of $u$ up to arbitrary orders.

When the background metric is not necessarily conformally flat, let $g$ be a smooth Riemannian metric on $B_1(0)\subseteq\mathbb R^n$. Consider
\begin{equation}
\label{eq-MEq5} \Delta_{g} u -\frac{n-2}{4(n-1)}S_gu=-\frac14n(n-2) u^{\frac{n+2}{n-2}} \quad\text{in }B_1(0)\setminus \{0\},
\end{equation}
where $ S_{g}$ is the scalar curvature of the metric $g$.
Let $u$ be a positive solution of \eqref{eq-MEq5} in $B_1(0)\setminus \{0\}$,
with a nonremovable singularity at the origin. In \cite{Marques2008}, Marques studied the asymptotic behavior of $u$, and proved that $u$ is still asymptotic to a radial singular solution of \eqref{eq-Yamabe} in $\mathbb R^n\setminus\{0\}$ when $3\leq n \leq 5$. Recently, the same result was proved for $n=6$ by Xiong and Zhang \cite{XZ} and for $7\leq n\leq24$ by Han, Xiong and Zhang \cite{HXZ}.
More precisely, when $3\leq n \leq 24$, it was proved that
\begin{equation}\label{eq-estimate-u-0}
|\frac{u(x)}{u_0(x)}-1|=|x|^{\alpha}\quad\text{as }x\to 0,\end{equation}
where $\alpha>0$ is some positive constant and $u_0(x)$ is a positive radial solution of \eqref{eq-Yamabe}
in $\mathbb R^n\setminus\{0\}$, with a nonremovable singularity at the origin.

Compare \eqref{main-estimate} and \eqref{main-estimate2}
with \eqref{eq-estimate-u-0}.

The paper is organized as follows.
In Section \ref{sec-pre}, we
introduce the concepts of Lipschitz domains locally bounded by finitely many
$C^{2}$-hypersurfaces and tangent cones at singular points on their boundaries.
In Section \ref{sec-Basic Estimates}, we study a class of second order elliptic partial differential operators and prove the existence and uniqueness of positive solutions for \eqref{eq-MEq1}-\eqref{eq-MBoundary2} in bounded Lipschitz domains in $(M,g)$.
In Section \ref{sec-eigenvalue}, we study the first eigenvalue and the corresponding eigenfunction of a specific
elliptic operator on spherical domains with coefficients singular on the boundary, which is used in Section \ref{sec-cone}-\ref{sec-domains locally bounded by hypersurfaces}.
In Section \ref{sec-cone},
we study asymptotic behaviors of positive solutions of \eqref{eq-MEq}-\eqref{eq-MBoundary} in domains which are locally Lipschitz cones and prove Theorem \ref{main reslut2}.
In Section \ref{sec-domains locally bounded by hypersurfaces}, we study the  asymptotic behaviors of positive solutions of \eqref{eq-MEq}-\eqref{eq-MBoundary} in domains which are locally bounded by $k$
linearly independent $C^2$ hypersurfaces and prove Theorem \ref{main reslut}.

\section{Preliminaries}\label{sec-pre}
In this section, we introduce the concepts of Lipschitz domains locally bounded by finitely many
$C^{2}$-hypersurfaces, tangent cones at singular points on their boundaries and  a $C^{2}$-diffeomorphism $T$.

\textit{First}, we discuss infinite cones and domains locally bounded by finite hypersurfaces. We emphasize
all cones in this paper are solid.

Let $V$ be an infinite cone in $\mathbb R ^{n}$ over a
Lipschitz domain $\Sigma\subsetneq\mathbb S^{n-1}$. In other words, $V\subsetneq \mathbb R ^{n}$ is an infinite cone with the origin being its vertex and $V\cap \mathbb{S}^{n-1}=\Sigma$ is a Lipschitz domain on $\mathbb S^{n-1}$.
According to \cite{HanJiangShen}, there exists a unique positive solution $u_V\in C^\infty(V)$ of \eqref{eq-MainEq}-\eqref{eq-MainBoundary} and
$u_V(x)=|x|^{-\frac{n-2}{2}} g(\theta)$, where $\theta\in \Sigma$ and $g$ satisfies
\begin{align}\label{eq-LN-spherical-domain-eq}
    \Delta_{\theta}g-\frac{1}{4}(n-2)^2g&=\frac{1}{4}n(n-2)g^{\frac{n+2}{n-2}}\quad\text{in } \Sigma,\\
\label{eq-LN-spherical-domain-boundary}   g&=\infty\quad\text{on }\partial \Sigma.
\end{align}

In particular, when $V=V_k \times \mathbb R^{n-k}\subsetneq\mathbb R^{n} $,
where $V_k$ is a cone over an appropriate domain $\Sigma_{k-1}\subsetneq \mathbb S^{k-1}$ in $\mathbb R^k$. Then, $u_V(x)$ has a form $|x'|^{-\frac{n-2}{2}}g_{k}(\theta)$,  where $\theta\in \Sigma_{k-1}$ and $|x'|=(x_1^2+...+x_k^2)^{\frac{1}{2}}$.

Here we note that if $V$ is an infinite Lipschitz cone in $\mathbb R ^{n}$ with vertex at $0$, then $\Sigma=V\bigcap \mathbb{S}^{n-1}$ is Lipschitz, but an infinite cone in $\mathbb R ^{n}$ over a
Lipschitz domain $\Sigma\subsetneq\mathbb S^{n-1}$ is not necessary Lipschitz. For example,  let $e_n$ be the north pole in $\mathbb S^{n-1}$ and $B_{r}^{g_{\mathbb S^{n-1}}}(e_n)$ be the ball on $\mathbb S^{n-1}$ with center  $e_n$, radius $r$. Let $\Sigma=B_{1}^{g_{\mathbb S^{n-1}}}(e_n)\setminus \overline{B_{1/2}^{g_{\mathbb S^{n-1}}}(e_n)}$ and $V$ be the infinite cone in $\mathbb R ^{n}$ over $\Sigma$. Then, $\Sigma$ is a Lipschitz domain on $\mathbb S^{n-1}$. However, $V$ is not Lipschitz near the origin.

\smallskip

We now give a definition of singular domains  locally bounded by finitely many $C^{2}$-hypersurfaces. For brevity, we focus on the case
when the boundary of the domain locally consists of
$k$ $C^{2}$-hypersurfaces intersecting at $0$ with the property that their  normal vectors
 at $0$ are linearly independent.  We adopt a definition basically the same as the definition in \cite{hanshen2}, cited as below.

\begin{definition}\label{def-Domain-lipschitz}
Let $S_1, \cdots, S_k$  be $k$ $C^{2}$-hypersurfaces passing $0$, with
the property that the normal vectors  of  $P_1, \cdots, P_k$, which are the tangent planes of
$S_1, \cdots, S_k$ at $0$, are mutually distinct. Let $\Omega$ be a bounded Lipschitz domain with
$0\in\partial\Omega$. Then, $\Omega \bigcap B_{R}(0)$ is called to be bounded by $S_1, \cdots, S_k$ in $B_{R}(0)$ for some $R>0$,
if,
$$\partial \Omega \bigcap B_{R}(0)\subseteq \bigcup_{i=1}^k  S_i.$$
We always assume $\overline{ \partial \Omega \bigcap B_{r}(0)\backslash S_i}\neq
\partial \Omega \bigcap \overline{B_{r}(0)}$, for any $r \leq R$ and any $i\in\{1,\cdots,k\}$. \end{definition}

Let $S_1, \cdots, S_k$, $P_1, \cdots, P_k$ and $\Omega$ be as in Definition \ref{def-Domain-lipschitz}. Then the tangent cone of $\Omega$ at $0$
is the solid infinite Lipschitz cone $V_0$ with the property:
$$\partial V_0\subseteq \bigcup_{i=1}^k  P_i,$$ and
$$\lim_{r\rightarrow 0}\frac{|\big[(\Omega\backslash V_0)\bigcup(V_0\backslash \Omega) \big]\bigcap B_r(0)  |}{r^n}=0.$$
The relation between these type of domains and their tangent cones can be described by satisfying a algebraic relation.
A detailed discussion of this characterization can be found in \cite{hanshen2}.

We now introduce the notation of signed distance. Choose a coordinate system such that  for some $R'\leq R$  and for each $i=1, \cdots, k$,
$S_i\cap B_{R'}(0)=\{x\in B_{R'}(0):\, x_n=f_i(x')\}$ for some $C^{2}$-function $f_i$ in $B_{R'}'(0')$. Then
\begin{equation}\label{bdylocalrep}\Omega\cap B_{R'}(0)=\{x\in B_{R'}(0)| x_n>f(x')\},\end{equation}
for some $f(x')\in \{f_1(x'), f_2(x'),\cdots,f_k(x') \}$.
Then,
$\partial \Omega\cap B_{R'}(0)=\{x\in B_{R'}(0):\, x_n=f(x')\}.$

Next, let $\nu_1, \cdots, \nu_k$ be the  unit normal vectors of
the tangent planes of $S_1, \cdots, S_k$ at $0$
such that, for any $i=1, \cdots, k$,
$$\langle \nu_i, e_n\rangle >0.$$

\begin{remark}
We choose the unit vector $e_n\in \mathbb{S}^{n-1}$ such that $$\min_{i=1,...n}\langle \nu_i, e_n\rangle =\max_{e_n\in \mathbb{S}^{n-1}}\min_{i=1,...n}\langle \nu_i, e_n\rangle.$$
For convenience, later on, we assume $R'=R$ and refer to $\sup_{(x', x_n)\in B_{r}(0)}|D^2f_i(x')|$ as the $C^2$ semi-norm of $S_i$ in $B_{r}(0)$ for any $r\in(0,R)$.

\end{remark}

Denote  the  distance from $x$ to $S_i$ with respect to the Euclidean metric $g_{E}$ by $d(x,S_i)$.
Then, the {\it signed distance} from $x$ to $S_i$ under the Euclidean metric with respect to $\nu_i$ near $0$ is defined by
\begin{equation}\label{oriented distance}
d_i(x)=
\begin{cases}
d(x,S_i)
& \text{if } x_n>f_i(x'),\\
0
& \text{if } x \in S_i,\\
-d(x,S_i)
& \text{if } x_n<f_i(x').\\
\end{cases}
\end{equation}
Under this setting, we always have $\{d_1>0, \cdots, d_n>0\}\bigcap B_r(0)\subseteq\Omega$ and $\{d_1<0, \cdots, d_n<0\}\bigcap B_r(0) \bigcap \Omega=\emptyset$ for $r>0$ sufficiently small.
Similarly, we can define the signed distance from $x$ to $P_i$ with respect to $\nu_i$.

Let $\Omega$ and $g$ be as in Assumption \ref{assumption-basic}. Denote the ordinary distance from $x$ to $S_i$ with respect to $g$ by $d_{g}(x,S_i)$.
Then, the {\it signed distance} from $x$ to $S_i$ with respect to $g$ and $\nu_i$ near $0$ is defined by
\begin{equation}\label{oriented distance}
d_{g,i}(x)=
\begin{cases}
d_g(x,S_i)
& \text{if } x_n>f_i(x'),\\
0
& \text{if } x \in S_i,\\
-d_g(x,S_i)
& \text{if } x_n<f_i(x').\\
\end{cases}
\end{equation}
For each $i$, $d_i(x)$ and $d_{g,i}(x)$ have the same sign.

With the notation of signed distance,
we can write the solution of
\eqref{eq-MainEq}-\eqref{eq-MainBoundary} for $\Omega=V_0$ in a form of $d_{P_i}$,  where $d_{P_i}$ is the signed distance form $x$ to $P_i$ with respect to $\nu_i$.
By a rotation,  we assume $V_0=
V_{k} \times \mathbb R^{n-k}$, with $V_k\subset \mathbb R^k$, where $1\leq k\leq n$.
Then for any $x $, the projection from $x$ to $\mathbb R^k\times\{0\}$ can
be uniquely determined by
$d_{P_1},...,d_{P_k}$.
With this
one-to-one correspondence between $x\in \mathbb R^k\times \{0\}$ and $(d_{P_1},...,d_{P_k})$,
the solution of
\eqref{eq-MainEq}-\eqref{eq-MainBoundary} for $\Omega=V_0$ can be expressed as
\begin{equation}\label{eq-Solution-Cone-d-coordinates-nd2}u_{V_0}(x)=f_{V_0}(d_{P_1}(x),...,d_{P_k}(x)).\end{equation}
In particular, when $k=1$, $f_{V_0}(d_1)=d_1^{-\frac{n-2}{2}}$.

\smallskip

\textit{Second}, we turn our attention to a $C^{2}$-diffeomorphism $T$, which was introduced in \cite{hanshen2} and  constructed as following.
If $k<n$, we add $n-k$ hyperplanes $P_{k+1}, \cdots, P_n$ which pass the point $0$
with their unit normal vectors $\nu_{k+1}, \cdots, \nu_n$ forming an orthonormal basis
of the orthogonal complement of Span$\{\nu_1,..,\nu_k\}$
and  denote by $d_{P_{j}}$ the signed distance from $x$ to $P_{j}$ with respect to $\nu_j$,
$j=k+1,\cdots,n$.
Then, for any $x$ sufficiently small, we define
$$T_{\{S\}}x=(d_{S_1}(x), \cdots,d_{S_k}(x), d_{P_{k+1}}(x), \cdots, d_{P_n}(x)),$$
where $d_{S_i}(x)$ is the signed distance from $x$ to $S_i$ with respect to $\nu_i$, $i=1, \cdots, k$.
$T_{\{S\}}$ is a $C^{2}$-diffeormorphism
in a neighborhood of the origin. A similar result holds for $T_{\{P\}}$, with $P_1, \cdots, P_n$ replacing
$S_1, \cdots, S_k,P_{k+1},..., P_n$.  Then, the map
\begin{equation}\label{C2diffeormorphism1}T=T^{-1}_{\{P\}}\circ T_{\{S\}}\end{equation} is a $C^{2}$-diffeomorphism
near the origin and has the property that the signed distance from $x$ to $S_i$ is the same as that from $Tx$
to $P_i$, for $i=1, \cdots, k$.

Similarly, let $\Omega$ and $g$ be as in Assumption \ref{assumption-basic}, we define $$T_{g,\{S\}}x=(d_{g,S}(x), \cdots,d_{g,S_k}(x), d_{g,P_{k+1}}(x), \cdots, d_{g,P_n}(x)),$$ where all signed distance are
with respect to $g$. Then
\begin{equation}\label{C2diffeormorphism} T_g=T^{-1}_{\{P\}}\circ T_{g,\{S\}} \end{equation}
is a $C^{2}$-diffeomorphism which maps
$B_r (0)$ to $T_g(B_r(0))\subseteq\mathbb R^n$, with
$T_g(\Omega \cap B_r (0))= V_{0}\cap T_g(B_r(0))$ and
$T_g(\partial\Omega \cap B_r (0))= \partial V_{0}\cap T_g(B_r(0))$ for some $r>0$.

For the $C^{2}$-diffeomorphism $T$, we give some computation which will be used in next sections.

 Without loss of generality, we assume $k=n$. Otherwise, we can add
$n-k$ hyperplanes as in the construction of $T$.
For a small positive constant $r,$
\begin{align*}
    T:\,&\Omega\cap B_r(0)\rightarrow\, V_0\\&x=(x_1,x_2,\cdots,x_n)\mapsto \bar{x}=(\bar{x}_1,\bar{x}_2,\cdots,\bar{x}_n)
\end{align*} such that $d_{S_i}(x)=d_{P_i}(\bar{x})=d_{P_i}(Tx)$ where $d_{S_i}$ is the signed distance to the $S_i$ and $d_{P_i}$ is the signed distance to the $P_i$ with respect to the Euclidean metric.

All vectors below are column vectors. For convenience, we use $d_i$ to denote $d_{S_i}$ and define
$n_i=\nabla_{x} d_i$ and $N=(n_1,n_2,\cdots,n_n)_{n\times n}.$ Then $$N^T(x)=\frac{\partial (d_1,\cdots,d_n)}{\partial (x_1,\cdots,x_n)},$$
abbreviated as $N^T$ sometimes.
 Let $$\bar{N}^T=\frac{\partial (d_1,\cdots,d_n)}{\partial (\bar{x}_1,\cdots,\bar{x}_n)},$$
and
$$(\bar{N}^T)^{-1}=\frac{\partial(\bar{x}_1,\cdots,\bar{x}_n)}{\partial (d_1,\cdots,d_n)}.$$
Then$$\frac{\partial(\bar{x}_1,\cdots,\bar{x}_n)}{\partial (x_1,\cdots,x_n)}=(\bar{N}^T)^{-1}N^T=(N\bar{N}^{-1})^T=(N(x)N(0)^{-1})^T.$$
In particular, we have
$$\frac{\partial(\bar{x}_1,\cdots,\bar{x}_n)}{\partial (x_1,\cdots,x_n)}|_{x=0}=I_{n\times n}.$$
Hence,
\begin{align}\label{diffeo-0}
JT(0)=I_{n\times n},
\end{align}
where $JT$ is the Jacobian matrix of $T$.

Now we consider the case that $\{S_i\}_{i=1,\cdots,n}$  are $C^2$ hypersurfaces near $x_0=0$ and  recall that we denote the unit normal vectors of $S_i$ at $0$ by $\nu_i$. Then
\begin{align}\label{n-nv}|n_i-\nu_i|\leq C|x-0|=C|x|.\end{align}
Hence,$$|\frac{\partial(\bar{x}_1,\cdots,\bar{x}_n)}{\partial (x_1,\cdots,x_n)}-I_{n\times n}|\leq C|x|,$$ and thus
\begin{align}\label{difxbarx}
|\bar{x}-x|\leq C|x|^2.
\end{align}
Then we make some useful computation for $f\in C^\infty (V_0).$
Take $f\circ T(x)=f(Tx)$. Then $f\circ T\in C^2\big(\Omega\cap B_r(0)\big)$ for some constant $r>0$, and
\begin{align*}
   \partial_{x_i}(f\circ T)(x)=\sum_{p,k=1}^n\partial_{\bar{x}_p}f(\bar{x})\partial_{d_k}\bar{x}_p(\bar{x})
   \partial_{x_i}d_k(x)=\sum_{k=1}^n\partial_{d_k}f(\bar{x})
   \partial_{x_i}d_k(x),
\end{align*} where $\bar{x}=Tx $ and the sum of $p$ and $ k$ are both taken from $1$ to $n.$
\begin{align*}
   \partial_{x_ix_j}^2(f\circ T)(x)=\sum_{k=1}^n\partial_{d_k}f(\bar{x})
   \partial_{x_ix_j}^2d_k(x)+\sum_{k,l=1}^n\partial_{d_kd_l}^2f(\bar{x})
   \partial_{x_i}d_k(x)\partial_{x_j}d_l(x).
\end{align*}
Since $S_i$ is $C^2$ near $0$, $\partial_{x_ix_j}^2d_k$ is bounded. Hence, by \eqref{n-nv},
\begin{align}\label{nablaf}\begin{split}
 |\partial_{x_i}(f\circ T)(x)-\partial_{\bar{x}_i}f(\bar{x})|
 =&\sum_{k=1}^n|\partial_{d_k}f(\bar{x})n_k^i-\partial_{d_k}f(\bar{x})\bar{n}_k^i|
\\ \leq &C|\nabla f(\bar{x})||x|,\end{split}\end{align}where $n_i=(n_i^1,\cdots,n_i^n)$ and
\begin{align}\label{hessf}
 |\partial_{x_ix_j}^2(f\circ T)(x)-\partial_{\bar{x}_i\bar{x}_j}^2 f(\bar{x})|
 \leq C(|\nabla f(\bar{x})|+|x||\nabla^2 f(\bar{x})|),\end{align}
 where $C$ only depends on $n$, $V_0$ and the
$C^{2}$-seminorms of $S_1, \cdots, S_n$ near $0.$
 In particular, we have
 \begin{align}\label{delerr}
    |\Delta_x(f\circ T)(x)-\Delta_{\bar{x}} f(\bar{x})|
     \leq C(|\nabla f(\bar{x})|+|x||\nabla^2 f(\bar{x})|).
 \end{align}

\section{ Existence and Uniqueness for \eqref{eq-MEq1}-\eqref{eq-MBoundary2} in Lipschitz domains in $(M, g)$}\label{sec-Basic Estimates}

In this section,
we consider a class of second order elliptic partial differential operators which can be viewed as suitable perturbation of the Euclidean Laplacian
operator and give some estimates for solutions of the corresponding perturbation elliptic equations and prove the existence and uniqueness of positive solutions for \eqref{eq-MEq1}-\eqref{eq-MBoundary2} in bounded Lipschitz domains.

\subsection{A class of second order differential operators on $(\mathbb R^n, g_E)$ }

First, we state a basic result that will be used later.

\begin{lemma}\label{lemma-supersution}
Let $ \Omega $ be a domain in $ \mathbb{R}^{n}$,
and $u$ and $v$ be two nonnegative solutions of \eqref{eq-MainEq}.
Then $u+v$ is a
nonnegative supersolution of \eqref{eq-MainEq}.
\end{lemma}

Next, we turn our attention to a class of second order elliptic operators,
Set
\begin{align}\label{L}
    L= \sum_{i,j=1}^n a_{ij}\partial_{ij}+\sum_{i=1}^n b_i\partial_i+c,
\end{align} where
$\big(a_{ij}\big)_{n\times n}$ is positive definite in $B_{2}(0)$,
 and $L$ satisfies the structure condition
\begin{align}\label{structurecondition}
     \sum_{i,j=1}^n|a_{ij}-\delta_{ij}|+ \sum_{i=1}^n|x||b_i|+|x|^2|c|\leq C_{L}|x|^2 \quad \quad \text{in } B_{2}(0),
\end{align}
where $C_{L}$ is a positive constant.

For any $R>0,$ set
\begin{align}\label{ur}
    u_R=(\frac{2R}{R^2-|x|^2})^\frac{n-2}{2}.
\end{align}
Then
\begin{align*}
    \partial_{x_i}u_R&=\frac{n-2}{2}(\frac{2R}{R^2-|x|^2})^{\frac{n-2}{2}
    -1}\frac{4Rx_i}{(R^2-|x|^2)^2},\\
     \partial^2_{x_ix_j}u_R&=\frac{1}{R}\frac{n-2}{2}(\frac{2R}{R^2-|x|^2})^{\frac{n-2}{2}
    +1}\delta_{ij}+\frac{1}{R}\frac{n-2}{2}\frac{n}{2}(\frac{2R}{R^2-|x|^2})^\frac{n-2}{2}
   \frac{4Rx_ix_j}{(R^2-|x|^2)^2} \\
      \Delta u_R&=\frac{n(n-2)}{4}[(\frac{2R}{R^2-|x|^2})^\frac{n-2}{2}]
      ^\frac{n+2}{n-2}.
\end{align*}

It is straight to verify that there exists a positive constant $R_{*}$ depending only on $n$ and $C_{L}$ such that when $R\leq R_*$,
\begin{align}\label{lur}\begin{split}
   L(2u_R)=&2\Delta u_R+2 \sum_{i,j=1}^n(a_{ij}-\delta_{ij})\partial^2_{ij}u_R+2 \sum_{i=1}^nb_i\partial_iu_R+2cu_R\\
   =&2\frac{n(n-2)}{4}u_R^\frac{n+2}{n-2}+O(|x|^2u_R^\frac{n+2}{n-2})+
   O(|x|u_R^\frac{n}{n-2}+O(u_R))\\
  < &\frac{n(n-2)}{4}(2u_R)^\frac{n+2}{n-2} \quad \text{in}\quad \Omega \cap B_{R}(0)\end{split}
\end{align}Hence, $2u_R$ is a super solution.

 Suppose $\Omega\subseteq\mathbb{R}^{n}$ is a Lipschitz domain with $0\in \partial\Omega$, $\partial\Omega \cap B_{2}(0)$ can be
 expressed by a Lipshitz function $x_n=f(x')$ in $B_{2}(0')$ and
\begin{equation}\label{omega-star}\Omega_*:=\Omega\cap B_{2}(0)=\{(x',x_n)\in B_{2}(0)|x_n>f(x')\},\end{equation}  where $f$ is Lipschitz with $f(0')=0.$

Suppose that $u_*$ is the positive solution to the following problem
\begin{align}\label{ustar}\begin{split}
   \Delta u_*=&\frac{n(n-2)}{4}u_*^\frac{n+2}{n-2}\quad \text{in}\quad \Omega\cap B_{1}(0),\\
   u_*=&+\infty\quad \text{on}\quad \partial(\Omega\cap B_{1}(0)).
\end{split}\end{align}

The following Lemma concerns properties of $u_*$ which was essentially derived from Lemma 3.2 and Lemma 3.4 in \cite{hanshen2}. The proofs there can be modified to yield this Lemma.

\begin{lemma}\label{ustar-estimates}
Let $ \Omega_{*} $ be a domain given in \eqref{omega-star} and $u_*$ be the solution of \eqref{ustar}.
Then, it holds in $\Omega_*\cap B_{\frac{1}{2}}(0),$
 \begin{equation}\label{uwtar-growth}C^{-1}\leq d^\frac{n-2}{2}u_*\leq 2^\frac{n-2}{2}\end{equation}
and
\begin{equation}\label{ustar-Derivative}d|\nabla u_*|+d^2|\nabla^2 u_*|\leq Cu_*
,\end{equation} where $d(x)$ is the Euclidean distance from
$x$ to $\partial\Omega_*\cap B_{1}(0)$, and $C$ is some constant depending only on $n$ and $\|f\|_{C^{0,1}(B_1(0'))}$.
\end{lemma}

Recall $L$ in \eqref{L}. Suppose that $u\in C^{2}(\Omega_*)$ is a positive solution to the following problem
\begin{align}\label{uo}\begin{split}
   L u=&\frac{n(n-2)}{4}u^\frac{n+2}{n-2}\quad \text{in}\quad \Omega_*,\\
   u=&+\infty\quad\text{on}\quad \{x_n=f(x')\}\cap B_{1}(0).\end{split}
\end{align}Note $\{x_n=f(x')\}\cap B_{1}(0)\subsetneq\partial\Omega_*$.

We have the following asymptotic estimates.
\begin{lemma}\label{A}
Let $\Omega$ be a Lipschitz domain and $\Omega_*$ be given in \eqref{omega-star}. Suppose $u\in C^{2}(\Omega_*)$ is a positive solution of \eqref{uo} and $u_*$ is the positive solution to \eqref{ustar}.
Then there exists a positive constant $\bar{R}$ depending only on $n$, $C_L$ and $\|f\|_{C^{0,1}(B_{1}(0'))}$ such that
\begin{align}\label{Ar}
   |u-u_*|\leq Cu_*|x|^\alpha\quad \text{in}\quad \Omega_*\cap B_{\bar{R}}(0),
\end{align} where $\alpha=2$ for $n\geq 6$, $\alpha=\frac{n-2}{2}$ for $n=3,4,5$, and $C$ is some constant depending only on $n$,
$C_L$ and $\|f\|_{C^{0,1}(B_{1}(0'))}$.
\end{lemma}

\begin{proof}

Take $$ \beta=\left\{
\begin{aligned}
&1-\frac{4}{n-2}=\frac{n-6}{n-2}\,\,\,&n\geq 6,\\
&0\,\,\,&n=3,4,5.
\end{aligned}
\right.
$$

Set $r=|x|$ and $w=u_*+Au_*^\beta+Bu_*r^2 $ where $A$ and $B$ are two positive constants to be determined and $v_*=u_*^{-\frac{2}{n-2}}.$
By \eqref{uwtar-growth},
\begin{align}\label{vd}
    v_*\sim d \quad\text{ near} \quad 0.
\end{align} Then
\begin{align}
    u_*^\beta\sim u_* v_*^\alpha \sim u_*d^\alpha\quad \text{near} \quad 0.
\end{align}
Next,
\begin{align}\label{uij}\begin{split}
   \partial_i u_*^\beta=&\beta u_*^{\beta-1}\partial_i u_*,\\
   \partial_{ij} u_*^\beta=&\beta u_*^{\beta-1}\partial_{ij} u_*+
   \beta(\beta-1)u_*^{\beta-2} \partial_i u_* \partial_j u_*\\
 \Delta u_*^\beta=&\beta u_*^{\beta-1}\Delta u_*+
   \beta(\beta-1)u_*^{\beta-2}|\nabla u_*|^2\\
   \leq &\beta u_*^{\beta-1}\Delta u_*,\\
      \partial_i (u_*r^2)=&r^2\partial_i u_*+2x_iu_*,\\
   \partial_{ij} (u_*r^2)=&r^2 \partial_{ij} u_*+
   2\partial_i u_*x_j+2\partial_j u_*x_i+2u_*\delta_{ij},\\
 \Delta (u_*r^2)=&r^2 \Delta u_*+4r\nabla r\nabla u_*+
  2nu_*.\end{split}
\end{align}Then
\begin{align*}
  & \Delta w-\frac{n(n-2)}{4}w^{\frac{n+2}{n-2}}\\= & \Delta w-\frac{n(n-2)}{4}u_*^{\frac{n+2}{n-2}}(1+Au_*^{\beta-1}+Br^2)^{\frac{n+2}{n-2}}
\\ \leq& \Delta w-\frac{n(n-2)}{4}u_*^{\frac{n+2}{n-2}}[1+\frac{n+2}{n-2}(Au_*^{\beta-1}+Br^2)]
\\ = &A\Delta (u_*^\beta)+B\Delta (u_*r^2)-
\frac{n(n+2)}{4}[Au_*^{\beta+\frac{4}{n-2}}+Br^2u_*^{\frac{n+2}{n-2}}]
\\ \leq& D_1+D_2+ D_3,\end{align*} where
\begin{align}\label{d1d2d3}\begin{split}
 D_1=&    -\frac{n(n+2)}{4}(1-\frac{n-2}{n+2}\beta)Au_*^{\beta+\frac{4}{n-2}}+A\beta(\beta-1)
u_*^{\beta-2}|\nabla u_*|^2,
\\  D_2=& -\frac{n(n+2)}{4}(1-\frac{n-2}{n+2}\beta)Br^2u_*^{\frac{n+2}{n-2}}
\\  D_3=& 4Br\nabla r\nabla u_*+B2n u_*.\end{split}
\end{align}
Note all the coefficients in $D_1$ and $D_2$ are negative and by Lemma  \ref{ustar-estimates} for some $C_1$,
\begin{align}\label{D3}
|D_3|\leq BC_1(\frac{u_*r}{d}+u_*).
\end{align}
Then
\begin{align}\label{lw}\begin{split}
   Lw-\frac{n(n-2)}{4}w^{\frac{n+2}{n-2}}=
   D_1+ D_2+ D_3+D_4,\end{split}
\end{align} where $D_4=O(|x|^2)|\nabla^2 w|+O(|x|)|\nabla w|+O(1)w.$

Now we choose $R$ small and $A,B$ big such that
\begin{align}\label{lwleq}
    Lw\leq \frac{n(n-2)}{4}w^{\frac{n+2}{n-2}}\quad \text{in}\quad \Omega_*\cap B_{R}(0).
\end{align}
We fix the constants in the  following order. \begin{itemize}
\item[Step 1] Take $R>0$ small enough compared to $A,B$ such that $Au_*^{\beta-1}+Br^2,$ for $ r\leq R,$ is small with $A$ and $B$ determined in next steps.
\item[Step 2] According to the powers of $r$ and $d$ in each term, we fix $B$
big such that $|\frac{1}{2}D_2|\geq|D_4|.$  In fact, to estimate $|\nabla^2 w|,\,\,|\nabla w|$ and $w,$ we use the computation \eqref{uij}, \eqref{uwtar-growth} and \eqref{ustar-Derivative}. From \eqref{uwtar-growth} and \eqref{ustar-Derivative}, we have
\begin{align*}
   |\nabla^2 u_*|\sim \frac{u_*}{d^2},
   \,\,|\nabla u_*|\sim \frac{u_*}{d},\,\,u_*^{\frac{n+2}{n-2}}\sim\frac{u_*}{d^2}.
\end{align*}
\item[Step 3] Fix $A$ big enough such that $|D_1+\frac{1}{2}D_2|\geq|D_3|.$
\end{itemize}Then by \eqref{lur} and \eqref{lwleq}, we have
\begin{align}\label{lurplw}\begin{split}
   L(w+2u_R)\leq
   &\frac{n(n-2)}{4}[w^\frac{n+2}{n-2}+(2u_R)^\frac{n+2}{n-2}]\\
   \leq &\frac{n(n-2)}{4}(w+2u_R)^\frac{n+2}{n-2}.\end{split}
\end{align}Set $F(\cdot)=\frac{n(n-2)}{4}(\cdot)^\frac{n+2}{n-2}.$ Then
\begin{align}\label{Atoo}
    Lu=F(u),\quad L(w+2u_R)\leq F(w+2u_R).
\end{align}

For any small $\epsilon>0,$ we take
\begin{equation}\label{u-epsilon}
u^\epsilon(x)=u(x_1,\cdots,x_{n-1},x_n+\epsilon ).
\end{equation}
By \eqref{Atoo}, at any $x\in\Omega\bigcap B_R(0)$,
\begin{align*}
    & \sum_{i,j=1}^n a_{ij}(w+2u_R-u^\epsilon)_{ij}+ \sum_{i=1}^nb_i(w+2u_R-u^\epsilon)_{i}\\
    \leq & F(w+2u_R)-F(u^\epsilon)-c(w+2u_R-u^\epsilon)\\
    \leq & F'(\xi_{x})(w+2u_R-u^\epsilon)-c(w+2u_R-u^\epsilon),
\end{align*} where we use the mean value theorem in the last equality and $\xi_{x}$ is between $w+2u_R$ and $u^\epsilon.$ Hence, $F'(\xi_{x})>0.$
Note $c$ in \eqref{L} is only bounded and it can change its sign. By \eqref{structurecondition}, $|c|\leq C_{L}$.

We will apply the maximum principle to prove
 \begin{align}\label{uwur}
 u\leq w+2u_R\quad \text{in}\quad\Omega\cap B_{R}(0),
\end{align} for small $R.$

First, we compare the boundary value. By \eqref{ur} and \eqref{u-epsilon}
\begin{align}\label{Abdv}
u^\epsilon<w+2u_R\quad \text{on} \quad\partial(\Omega\cap B_{R}(0)).
\end{align}
We will prove
\begin{align}\label{uepsilonwur}
u^\epsilon\leq w+2u_R\quad \text{in}\quad\Omega\cap B_{R}(0)
\end{align}
 and finally let $\epsilon$ go to $0.$

 First, we take $R>0$ small such that $w+2u_R$ is large and
 \begin{align}\label{bd1}
  \frac{n(n+2)}{4}(w+2u_R)^\frac{4}{n-2}>C_L\quad \text{in}\quad B_{R}(0)\cap\Omega_*.
 \end{align}

Next, we discuss in $G=B_{R}(0)\cap\Omega\cap\{u^\epsilon>\big(\frac{4C_L}{n(n+2)}\big)^{\frac{n-2}{4}}\}.$

We compare the values of $F'(\xi_x)$ and $C_L$. In $G,$ $$\frac{n(n+2)}{4}(u^\epsilon)^\frac{4}{n-2}>C_L.$$
Then, by \eqref{bd1}, $F'(\xi_x)>C_L.$

On $\partial G,$ by \eqref{Abdv} and \eqref{bd1}, $u^\epsilon\leq w+2u_R. $

 Then we apply the maximum principle to $w+2u_R-u^\epsilon$ in $G$ and have $$w+2u_R-u^\epsilon\geq 0\quad\text{in }\quad G.$$

On the other hand, by \eqref{bd1},
\begin{align*}
    u^\epsilon\leq (\frac{4C_L}{n(n+2)})^{\frac{n-2}{4}}\leq w+2u_R\quad \text{in} \quad B_{R}(0)\cap\Omega\cap G^c.
\end{align*}
Then we have \eqref{uwur} by letting $\epsilon$ go to $0.$

Next, we note, in $B_{\frac{R}{2}}(0),$ it holds $u_{R}\leq C(n,R)$ and
by \eqref{uwtar-growth}, $u_*^{-1}\sim d^{\frac{n-2}{2}}.$
Then we have, for $\bar{R}\leq \frac{R}{2},$ by \eqref{uwur}, $$u-u_*\leq Cu_*|x|^\alpha+ Cu_*u_*^{-1}=Cu_*(|x|^\alpha+d^{\frac{n-2}{2}})\quad \text{in}\quad \Omega_*\cap B_{\bar{R}}(0).$$

For the other direction, the proof is similar and we only point out the key changes.
Set $w_1=u_*-A_1u_*^\beta-B_1u_*r^2 $ where $A_1$ and $B_1$ are two positive constants to be determined.
Then
\begin{align*}
  & \Delta w_1-\frac{n(n-2)}{4}w_1^{\frac{n+2}{n-2}}\\= & \Delta w_1-\frac{n(n-2)}{4}u_*^{\frac{n+2}{n-2}}(1-A_1u_*^{\beta-1}-B_1r^2)^{\frac{n+2}{n-2}}
\end{align*}
As mentioned in step 1, we take $R_1>0$ small enough compared to $A_1,B_1$ such that $A_1u_*^{\beta-1}+B_1r^2$, for $r\leq R_1,$ is sufficiently small. Then\begin{align*}
  & \Delta w_1-\frac{n(n-2)}{4}w_1^{\frac{n+2}{n-2}}\\ \geq & \Delta w_1-\frac{n(n-2)}{4}u_*^{\frac{n+2}{n-2}}[1-(1-c(n))\frac{n+2}{n-2}(A_1u_*^{\beta-1}+B_1r^2)]
\\ = &-A_1\Delta  (u_*^\beta)-B_1 \Delta (u_*r^2)+
(1-c(n))\frac{n(n+2)}{4}[A_1u_*^{\beta+\frac{4}{n-2}}+B_1r^2u_*^{\frac{n+2}{n-2}}]
\\ \geq&D_1+D_2+ D_3,\end{align*} where $c(n)$ is a small constant depending on $n$ such that $1-c(n)-\frac{n-2}{n+2}\beta>0,$ which is guaranteed by the smallness of $R_1$, and
\begin{align}\begin{split}
 D_1=&    \frac{n(n+2)}{4}(1-c(n)-\frac{n-2}{n+2}\beta)A_1u_*^{\beta+\frac{4}{n-2}}-A_1\beta(\beta-1)
u_*^{\beta-2}|\nabla u_*|^2,
\\  D_2=& \frac{n(n+2)}{4}(1-c(n)-\frac{n-2}{n+2}\beta)B_1r^2u_*^{\frac{n+2}{n-2}}
\\  D_3=& -4B_1r\nabla r\nabla u_*-B_12n u_*.\end{split}
\end{align} Then we can proceed as before. We have
 \begin{align*}
 w_1\leq u+2u_R\quad \text{in}\quad\Omega\cap B_{R}(0),
\end{align*}

\end{proof}

\begin{remark}\label{rek-sec3}
Suppose that $u_{**}$ is a positive solution to the following problem
\begin{align*}
   \Delta u_{**}=&\frac{n(n-2)}{4}u_{**}^\frac{n+2}{n-2}\quad \text{in}\quad \Omega\cap B_{1}(0),\\
   u_{**}=&+\infty\quad \text{on}\quad \partial\Omega\cap B_{1}(0).
\end{align*}
If we replace $u_{*}$ in \eqref{Ar} with $u_{**}$, then Lemma \ref{A} still holds.

\end{remark}
 The following corollary is a direct consequence of Lemma \ref{A} and the estimate \eqref{uwtar-growth} of $u_{*}$.
\begin{corollary}\label{corollary-estimate}
Let $\Omega$ be a Lipschitz domain and $\Omega_*$ be given in \eqref{omega-star}. Suppose that $u\in C^{2}(\Omega \bigcap B_1(0))$ is a positive solution of \eqref{uo}. Then there exists a positive constant $\delta$,
depending only on $n$, $C_L$ and $\|f\|_{C^{0,1}(B_{1}(0'))}$
such that,
\begin{equation}
C^{-1}
\leq d(x)^{\frac{n-2}{2}}u(x) \leq C \quad\quad\text{in}\quad \Omega\cap B_{\delta}(0), \end{equation}  where $d(x) $ is the distance from
$x$ to $\partial\Omega$, and $C$ is a positive constant depending only on $n$, $C_L$ and $\|f\|_{C^{0,1}(B_{1}(0'))}$.
\end{corollary}

\subsection{Existence and uniqueness results in $(M, g)$ }
We now prove the existence and uniqueness of positive solutions for \eqref{eq-MEq1}-\eqref{eq-MBoundary2} in bounded Lipschitz domains in $(M, g)$.

\begin{theorem}\label{thrm-Existence}
Suppose that $(M, g)$ is a compact Riemannian manifold without boundary of dimension $n\geq 3$.
Let $ \Omega \varsubsetneq M$ be a Lipshitz domain.
Then \eqref{eq-MEq1}-\eqref{eq-MBoundary2}
admits a unique positive solution $ u \in C^{ \infty}(\Omega)$.
\end{theorem}

\begin{proof}
Without loss of generality, we assume $S_g$ is constant on $M$.

First, we prove the existence. If $\Omega$ is a smooth domain, according to \cite{AM1988DUKE}, or the recent work \cite{L},
\eqref{eq-MEq1}-\eqref{eq-MBoundary2} admits a unique positive solution.
Let $\{\Omega_i\}$ be a sequence of decreasing smooth domains in $M$ which converges to $\Omega$ with $\{\partial \Omega_i\}$ possessing uniform Lipschitz constant, and $u_i$ be the positive solution of \eqref{eq-MEq1}-\eqref{eq-MBoundary2}
in $\Omega_i$. If $S_g\geq0$, we can apply the maximum principle.
If $S_{g}<0$, by $(4.3)$ in \cite{HanShen3}:
\begin{equation}\label{eq-AlgebraicRelation}u_i\geq
\Big(\frac{-S_g}{n(n-1)}\Big)^{\frac{n-2}{4}}
\quad\text{in }\Omega_i.\end{equation}
Then we can also apply the maximum principle to obtain $u_i\leq u_{i+1}$ in $\Omega$.  For any $x\in\Omega$, take $ \varepsilon>0$
small such that $\overline{B_{\varepsilon}(x)} \subseteq\Omega$. Let $u_x$ be the positive solution of \eqref{eq-MEq1}-\eqref{eq-MBoundary2} in $B_{\varepsilon}(x)$. By the maximum principle, for any $i$, $u_{i}(x)\leq u_x(x)$.

Then, it is straightforward to verify, for any positive integer $m$,
$$u_i\rightarrow u\quad\text{in }C^{m}_{\mathrm{loc}}( \Omega)
\text{ as }i\rightarrow\infty,$$
for some $u\in C^{\infty}(\Omega)$, $u>0$ in $\Omega$ and $u$ satisfies \eqref{eq-MEq1}. The boundary condition \eqref{eq-MBoundary2} can be verified by using Corollary \ref{corollary-estimate}.

Next, we prove the uniqueness. Let $v$ be a positive solution of \eqref{eq-MEq1}-\eqref{eq-MBoundary2}. By the maximum principle,
$u_i\leq v$ in $\Omega$. Hence, $u\leq v$ in $\Omega$. Set $w=\frac{v}{u}-1$. Then, $w\geq0$ in $\Omega$ and $w$ solves
 $$\Delta_{g} w+2\frac{\nabla_g w \nabla_g u}{u}= \frac14n(n-2) u^{\frac{4}{n-2}}\big[(1+w)^{\frac{n+2}{n-2}}-(1+w)\big].$$
 Also, we can check $w(x)\rightarrow0$ as $x\rightarrow\partial\Omega$
by Lemma \ref{A}.

By the strong maximum principle, we have $w=0$ in $\Omega$.
\end{proof}

\section{Eigenvalue and Eigenfunction}\label{sec-eigenvalue}
In this section, we study the first eigenvalue and the corresponding eigenfunction of a specific
elliptic operator on spherical domains with coefficients singular on the boundaries.
 The lower bound of the first eigenvalue plays an important role in this paper.

Let ${V}$ be a Lipschitz infinite cone in $\mathbb R ^{n}$ over a domain $\Sigma\subsetneq\mathbb S^{n-1}$.
Let $u_V\in C^\infty(V)$ be the unique positive solution of \eqref{eq-MainEq}-\eqref{eq-MainBoundary}. Then, by \cite{HanJiangShen}, $$ u_{V}=|x|^{-\frac{n-2}{2}}g(\theta)=
r^{-\frac{n-2}{2}}g(\theta),$$ where $g>0$ in $V\cap \mathbb{S}^{n-1}=\Sigma$ is the solution of \eqref{eq-LN-spherical-domain-eq}-\eqref{eq-LN-spherical-domain-boundary}.
Set
$\rho=g^{-\frac{2}{n-2}}$. Then
\begin{align}\label{uvrphojili66}
u_{V}=r^{-\frac{n-2}{2}}\rho^{-\frac{n-2}{2}}.
\end{align}
It was proved in \cite{HanJiangShen} that $\rho\in C^\infty(\Sigma)\cap \mathrm{Lip}(\Sigma)$ and $|\nabla_\theta \rho|\le C$, where the constant $C$ depends only on $n$ and the size of exterior cone of $V$, and hence $C$ depends only on $n$ and $\Sigma$. Denote $d_{g_{\mathbb{S}^{n-1}}}(x,\partial \Sigma)$ by $d_{\Sigma}$.
Then by Lemma 2.4 in \cite{HanJiangShen}, $\rho$ satisfying
\begin{equation}\label{eq-estimate-rho-upper-lower}
c_1\le \frac{\rho}{d_{\Sigma}}\le c_2\quad\text{on }\Sigma,\end{equation}
for some positive constants $c_1$ and $c_2$ depending only on $n$ and $\Sigma$.

We now consider the eigenvalue problem of a specific
elliptic operator on spherical domains with coefficients singular on the boundary which is related to the linearized operator of the equation  \eqref{eq-MainEq}.

By \cite{HanJiangShen}, there exists a function $\phi_1\in H_0^1(\Sigma)\cap C^\infty(\Sigma)\cap C(\overline{\Sigma})$ satisfying $\phi_1>0$ in $\Sigma$ with $\|\phi\|_{L^2(\Sigma)}=1$ and
\begin{equation}\label{eq-def-L-Sigma}
L_{\Sigma}\phi_1:=-\Delta_\theta \phi_1+\frac{n(n+2)}{4\rho^2}\phi_1=\lambda_1 (L_{\Sigma})\phi_1,\end{equation}
in $\Sigma$, where $\Delta_\theta$ is the Laplacian operator with respect to $g_{\mathbb{S}^{n-1}}$ and
\begin{align}\label{1st-eigenvalue}\begin{split}
\lambda_1(L_{\Sigma})=&\inf_{\phi\in H_0^1(\Sigma),\phi\neq 0}\frac{\int_{\Sigma}\big(|\nabla_\theta \phi|^2+\frac{n(n+2)}{4\rho^2}\phi^2\big)d\theta}{\int_{\Sigma}\phi^2 d\theta}\\
=&\inf_{\phi\in H_0^1(\Sigma),\|\phi\|_{L^2(\Sigma)}=1}\int_{\Sigma}\big(|\nabla_\theta \phi|^2+\frac{n(n+2)}{4\rho^2}\phi^2\big)d\theta
\end{split}\end{align}
is the first eigenvalue of $L_{\Sigma}$.

Concerning the derivatives of $\phi_1$, we have the following lemma.

\begin{lemma}\label{lemma-phi-estimate}
Let $\Sigma$ be a Lipschitz domain in $\mathbb{S}^{n-1}$.
Suppose $\phi_1\in H_0^1(\Sigma)\cap C^\infty(\Sigma)\cap C(\overline{\Sigma})$ is positive in $\Sigma$ and satisfies
\eqref{eq-def-L-Sigma}.
Then, there exists a constant $\nu>0$ depending only on $n$ and $\Sigma$ such that
\begin{equation}\label{eq-Derivative}
\rho(x)|\nabla_{\theta} \phi_1(x)|+\rho^2(x)|\nabla_{\theta}^2 \phi_1(x)|\le C\rho^{\nu},\end{equation} where $C$ is a constant depending only on $n$ and $\Sigma$.
\end{lemma}

\begin{proof}

By Theorem 4.4 in \cite{HanJiangShen},   $\phi_1\le C\rho^{\nu}\leq C d_{\Sigma}^{\nu}$ for some positive constant $\nu$ depending only on $n$ and $\Sigma$.
 We take an arbitrary $\beta\in (0,1)$. By the standard
interior estimates, we have, for any $x\in \Sigma$,
 \begin{align*}&d_{\Sigma}(x)^{\beta}[\phi_1]_{C^{\beta}(B_{{d_{\Sigma}}/{4}}(x))}+d_{\Sigma}(x)[\nabla_{\theta} \phi_1]_{L^{\infty}(B_{{d_{\Sigma}}/{4}}(x))}
+d_{\Sigma}(x)^{1+\beta}[\nabla_{\theta} \phi_1]_{C^{\beta}(B_{{d_{\Sigma}}/{4}}(x))}\\
&\leq
C\{|\phi_1|_{L^{\infty}(B_{{d_{\Sigma}}/{2}}(x))}
\}\le C\rho^{\nu},\end{align*}
and
$$\aligned &d_{\Sigma}(x)^{2}|\nabla^{2}_{\theta}\phi_1|_{L^{\infty}(B_{{d_{\Sigma}}/{8}}(x))}+d_{\Sigma}(x)^{2+\beta}[\nabla^2_{\theta} \phi_1]_{C^{\beta}(B_{{d_{\Sigma}}/{8}}(x))}\\
&\qquad\leq
C\{|\phi_1|_{L^{\infty}(B_{{d_{\Sigma}}/{4}}(x))}+d_{\Sigma}(x)^{\beta}[\phi_1]_{C^{\beta}(B_{{d_{\Sigma}}/{4}}(x))}\leq C\rho^{\nu},\endaligned$$
where $C$ is a positive constant depending only on $n$, $\beta$ and $\Sigma$.
\end{proof}

By a direct computation, we have the following estimates.

\begin{corollary}
Let ${V}$ be a Lipschitz infinite cone in $\mathbb R ^{n}$ over a domain $\Sigma\subsetneq\mathbb S^{n-1}$.
Then for any constant $\alpha>0$ and $x\in V\cap B_1(0)$, we have
\begin{align*}
|\nabla(|x|^{\alpha}\phi_1)|&\leq C |x|^{\alpha-1}\rho^{\nu-1},\\
|\nabla^2(|x|^{\alpha}\phi_1)|&\leq  C |x|^{\alpha-2}\rho^{\nu-2},
\end{align*}
where $C$ is some constant depending only on $n$, $\alpha$ and $\Sigma$.
\end{corollary}

The following are some properties of $\lambda_1(L_{\Sigma})$.
\begin{prop}\label{eigen-comparision}
Let $\Sigma_1, \Sigma_2 \subsetneq\mathbb S^{n-1}$ be two Lipschitz domains with $\Sigma_1\subsetneqq\Sigma_2 $. Then
$\lambda_1(L_{\Sigma_1})> \lambda_1(L_{\Sigma_2}).$
\end{prop}

\begin{proof}
Let $g_{i}$ be the solution of \eqref{eq-LN-spherical-domain-eq}-\eqref{eq-LN-spherical-domain-boundary} for $\Sigma=\Sigma_i$, $i=1,2$.
Set $\rho_i=g_i^{-\frac{2}{n-2}}$, $i=1,2$.
By the strong maximum principle, we have $g_{1}(\theta)>g_{2}(\theta)$ in $\Sigma_1$ and hence $\rho_{1}(\theta)<\rho_{2}(\theta)$. Then the desired result follows by \eqref{1st-eigenvalue}.

\end{proof}

In order to estimate the lower bound of $\lambda_1(L_{\Sigma})$. We need the following Lemma.

\begin{lemma}\label{gap}
Let $V$ be a Lipschitz infinite cone in $\mathbb R ^{n}$ over a domain $\Sigma\subsetneq\mathbb S^{n-1}$,  $u_{V}$ be the solution
of \eqref{eq-MainEq}-\eqref{eq-MainBoundary} for $\Omega=V$ and $\underline{u}$ be the solution
of \eqref{eq-MainEq}-\eqref{eq-MainBoundary} for $\Omega=\mathbb R^{n}\setminus \big(\overline{V^c\cap B_1(0)}\big)$. Then, there exists a positive
constant $c$ such that
 \begin{equation*}
\underline{u} \leq u_{V}-c|x|^{-\frac{n-2}{2}+\mu_1}\phi_1,   \quad\text{in }\, V\cap  B_{2}(0),
\end{equation*}
where $\mu_1=\sqrt{(\frac{n-2}{2})^2+\lambda_1(L_{\Sigma})}$ and $\phi_1>0$ is the first eigenfunction of $L_{\Sigma}$.
\end{lemma}
\begin{proof}

By Lemma \ref{lemma-supersution}, $\underline{u}\leq u_{V}+\Big(\frac{2r}{r^{2}-|x|^{2}}\Big)^{\frac{n-2}{2}}$ in $V\cap B_r(0)$.
Let $r\rightarrow \infty$, the strong maximum principle implies
\begin{equation*}
\underline{u}<  u_{V} \quad\text{in }\, V. \end{equation*}
By a direct computation, we have
\begin{equation*}
 \Delta \big(|x|^{-\frac{n-2}{2}+\mu_1}\phi_1\big)=\frac{n(n+2)}{4} u_{V}^{\frac{4}{n-2}}(|x|^{-\frac{n-2}{2}+\mu_1}\phi_1\big)
\end{equation*}
Hence, fixing any $\epsilon>0$ sufficiently small, we have
 \begin{equation*}
 \Delta\big(u_{V}-\epsilon|x|^{-\frac{n-2}{2}+\mu_1}\phi_1)< \frac{n(n-2)}{4}\big(u_{V}-\epsilon|x|^{-\frac{n-2}{2}+\mu_1}\phi_1)^{\frac{n+2}{n-2}},   \quad\text{in }\, V\cap B_{2}(0)
\end{equation*}
Take $c>0$ smaller than $\epsilon$ such that
 \begin{equation*}
\underline{u} < u_{V}-c|x|^{-\frac{n-2}{2}+\mu_1}\phi_1,   \quad\text{on }\, V\cap \partial B_{2}(0).
\end{equation*}
Then by the maximum principle,
 \begin{equation*}
\underline{u} \leq u_{V}-c|x|^{-\frac{n-2}{2}+\mu_1}\phi_1  \quad\text{in }\, V\cap B_{2}(0).
\end{equation*}

\end{proof}

\begin{prop}\label{eigen-esti}
If $\Sigma\subsetneq\mathbb S^{2}$ is a Lipschitz domain, then
$\lambda_1(L_{\Sigma})>\frac{3}{4}.$
\end{prop}

\begin{proof}
We prove by contradiction.
Suppose $\Sigma\subsetneq\mathbb S^{2}$ is a Lipschitz domain with
$\lambda_1(L_{\Sigma})\leq \frac{3}{4}.$
Let $V\subsetneq\mathbb R^{3}$  be a convex infinite cone enclosed by three linearly independent hyperplanes passing through the origin
and satisfying $V\bigcap\mathbb S^{2} \subsetneqq \mathbb S^{2}\setminus \overline{\Sigma}$. Set $\overline{V}^c=\mathbb R^{3}\setminus \overline{V}$,
$\Sigma_{\overline{V}^c}= \overline{V}^c\bigcap\mathbb S^{2}.$ Then, by Proposition \ref{eigen-comparision}, $\lambda_1(L_{\Sigma_{\overline{V}^c}})<\frac{3}{4}.$

Let $u_{\overline{V}^c}=|x|^{-\frac{1}{2}}g_{\overline{V}^c}$ be the solution
of \eqref{eq-MainEq}-\eqref{eq-MainBoundary} for $\Omega=\overline{V}^c$ and $u$ be the solution
of \eqref{eq-MainEq}-\eqref{eq-MainBoundary} for $\Omega=\mathbb R^{3}\setminus \big(\overline{V\bigcap B_1(0)}\big)$.

Set $\mu_1=\sqrt{1/4+\lambda_1(L_{\Sigma_{\overline{V}^c}}) }<1$. By Lemma \ref{gap}, there exists $c>0$ such that
\begin{equation*}
u \leq u_{\overline{V}^c}-c|x|^{-\frac{1}{2}+\mu_1}\phi_1= u_{\overline{V}^c}[1-c(g_{\overline{V}^c}^{-1}\phi_1)|x|^{\mu_1}]  \quad\text{in }\, \overline{V}^c\cap B_{2}(0),
\end{equation*}
where $\phi_1>0$ is the first eigenfunction of $L_{\Sigma_{\overline{V}^c}}$.

However, the result in \cite{hanshen2} shows  \begin{equation*}
u =u_{\overline{V}^c}[1+O(|x|)]\quad\text{in }\, \overline{V}^c\cap B_{1}(0),
\end{equation*}
which leads to a contradiction!
We point out that although Theorem 1.1 in \cite{hanshen2} is stated in the case of bounded domains, its proof is also applicable to the present case.

\end{proof}

When $n\geq4$, The following proposition indicates that $\inf_{\text{Lipschitz } \Sigma\subsetneq\mathbb S^{n-1}}\lambda_1(L_{\Sigma})=0$.
\begin{prop}\label{eigen-esti2}
Suppose $n\geq4$, $p$ is a point in $\mathbb S^{n-1}$ and $B_{r}(p)$
 is a small ball on $\mathbb S^{n-1}$ with center $p$, radius $r$. Denote $\mathbb S^{n-1}\setminus \overline{B_{r}(p)}$ by $\Sigma_{r}$. Then,
$$\lim_{r\rightarrow0}\lambda_1(L_{\Sigma_{r}})=0.$$
\end{prop}

\begin{proof}
Without loss of generality, we assume $p=(0,...,0,1)$.
Let $V_r$ be the Lipschitz infinite cone in $\mathbb R ^{n}$ over $\Sigma_r$.
Then, $V_r\cap B_{\frac{1}{r}}(0)\rightarrow\mathbb R ^{n}\setminus \gamma$ as $r\rightarrow0$, where $\gamma=\{(0, \cdots, 0, x_n)|\, x_n>0\}\subset\mathbb R^n.$
Let $u_r$ be the solution of \eqref{eq-MainEq}-\eqref{eq-MainBoundary} for $\Omega=V_r\bigcap B_{\frac{1}{r}}(0)$ and $g_r$ be the solution of \eqref{eq-LN-spherical-domain-eq}-\eqref{eq-LN-spherical-domain-boundary} in $\Sigma_r$.

Arguing as Example 5.6 in \cite{HanShen3}, using the inverse transform of the stereographic projection
which lifts $\mathbb{R}^{n}\times \{0\}$ to $\mathbb S^{n}$, and applying Lemma 4.1 in \cite{HanShen3}, we have
\begin{equation*}
u_r\rightarrow 0
\quad\text{in }C_{\mathrm{loc}}( \mathbb R^{n} \setminus\{\gamma\})
\text{ as }r\rightarrow0.\end{equation*}
By the maximum principle, $ |x|^{-\frac{n-2}{2}}g_r \leq u_r$ in $V_r\bigcap B_{\frac{1}{r}}(0)$.
Hence,
\begin{equation*}
g_r\rightarrow 0
\quad\text{in }C_{\mathrm{loc}}( \mathbb S^{n-1} \setminus\{p\})
\text{ as }r\rightarrow0.\end{equation*}
Set $\rho_r=g_r^{-\frac{2}{n-2}}$, we have
\begin{equation*}
\rho_r^{-2}\rightarrow 0
\quad\text{in }C_{\mathrm{loc}}( \mathbb S^{n-1} \setminus\{p\})
\text{ as }r\rightarrow0.\end{equation*}

For any fixed $s>0$ sufficiently small, when $r<s$,
take
\begin{equation*}
\phi_s=
\begin{cases}
0
& \text{if }  d_{\mathbb S^{n-1}}(x,p)\leq s,\\
\frac{d_{\mathbb S^{n-1}}(x,p)-s}{s}
& \text{if }  s<d_{\mathbb S^{n-1}}(x,p)\leq 2s,\\
1
& \text{if } d_{\mathbb S^{n-1}}(x,p)\geq 2s.\\
\end{cases}
\end{equation*}
Substitute $\phi_s$ into the right hand side of \eqref{1st-eigenvalue}, we have
$$\limsup_{r\rightarrow0}\lambda_1(L_{\Sigma_{r}})\leq Cs^{n-3},$$
for some positive constant $C$ depending only on $n$. Then the desired result follows directly.

\end{proof}

\section{Solutions in finite Cones}\label{sec-cone}

In this section, we study the asymptotic behaviors of positive solutions of \eqref{eq-MEq}-\eqref{eq-MBoundary} in domains which are locally Lipschitz cones and prove Theorem \ref{main reslut2}.

First, we will improve the estimates derived in Lemma \ref{A} when $\Omega$ is locally a cone.
 For simplicity, in rest of this paper we will abbreviate $\lambda_1(L_{\Sigma})$ to $\lambda_1$ and let \begin{equation}\mu_1=\sqrt{(\frac{n-2}{2})^2+\lambda_1}.  \end{equation}
 By Proposition \ref{eigen-esti}, we always have
 \begin{align}\label{Kmu1-19}
 \mu_1> \max\{\frac{n-2}{2},1\}.
 \end{align}

\begin{theorem}\label{thm-cone}
 Let $V$ be a Lipschitz infinite cone in $\mathbb R ^{n}$ over a domain $\Sigma\subsetneq\mathbb S^{n-1}$. Suppose $u\in C^{\infty}(V\bigcap B_{1}(0))$ is a positive solution of
 \begin{align}\label{Lu}\begin{split}
   L u=&\frac{n(n-2)}{4}u^\frac{n+2}{n-2}\quad in\quad V\cap B_{1}(0),\\
   u=&+\infty\quad on\quad \partial V\cap B_{1}(0),\end{split}
\end{align}
and $u_{V}=|x|^{-\frac{n-2}{2}}g(\theta)$ is the solution of \eqref{eq-MainEq}-\eqref{eq-MainBoundary} for $V$, where $L$ is the second order elliptic operator given by
\eqref{L} and satisfies \eqref{structurecondition}.
Then, there exists a positive constant $R\leq 1/2$, depending only on $n$, $\Sigma$ and $C_{L}$ in \eqref{structurecondition}, such that,
for any $x\in  V \cap B_{R}(0)$,
\begin{equation}\label{esti-u-cone}
\Big|\frac{u(x)}{u_{V}(x)}-1\Big|\leq
\begin{cases}
C|x|^2
& \text{if } \mu_1>2,\\
C\big(-|x|^2\ln|x|\big)
& \text{if } \mu_1=2,\\
C|x|^{\mu_1}
& \text{if } \mu_1<2,\\
\end{cases}
\end{equation}
where $C$ is a positive constant depending only on $n$, $\Sigma$ and $C_{L}$ in \eqref{structurecondition}.
\end{theorem}

\begin{proof}

When $n\geq 6,$ the results follow from Lemma \ref{A} since in this case $\mu_1>\frac{n-2}{2}\geq 2$.
So we only need to consider $n=3,4,5.$

Recall $\rho=g^{-\frac{2}{n-2}}$ and $r=|x|$.
Let $\alpha>0$, $A_1$ and $A_2$ be two positive constants. A straightforward calculation yields
\begin{equation}\label{est-1}
\Delta (A_1r^{\alpha}+A_2r^{\alpha}\phi_1)=r^{\alpha-2}\big[\alpha(\alpha+n-2)(A_1+A_2\phi_1)+A_2\big(\frac{n(n+2)}{4}\rho^{-2}\phi_1-\lambda_1\phi_1\big)\big],
 \end{equation}
 and
\begin{equation}\label{est-2}
\Delta (u_{V}r^{\alpha+\frac{n-2}{2}})
=\frac{n(n-2)}{4}u_{V}^{\frac{n+2}{n-2}}r^{\alpha+\frac{n-2}{2}}+(\alpha+\frac{n-2}{2})^2u_{V}r^{\alpha+\frac{n-2}{2}-2},
 \end{equation}
where $\phi_1>0$ with $\|\phi\|_{L^2(\Sigma)}=1$ is the first eigenfunction of $L_{\Sigma}$.

By \eqref{eq-estimate-rho-upper-lower}, (Lemma 3.2, Remark 3.3, Lemma 3.4) in \cite{hanshen2}, we have

\begin{equation}\label{esti-u-A}
u_{V}+r\rho\sum_{i=1}^{n}|\partial_i u_{V}(x)|+(r\rho)^2\sum_{i,j=1}^{n}|\partial_{ij}^2u_{V}(x)|\le A u_{V} \quad \quad\text{in }V.\end{equation}
where $A>1$ is a constant depending only on $n$ and $\Sigma$.

By Lemma \ref{A} and Remark \ref{rek-sec3}, there exists a positive constant $r_1$, such that
\begin{equation}\label{estim-2-B}|u-u_{V}|\leq B u_{V}|x|^{\frac{n-2}{2}}, \quad \quad\text{in }V\cap B_{r_1}(0)\end{equation}
 for some constant $B$ depending only on $n$, $C_L$ and $\Sigma$.
\smallskip

We discuss in three cases.

{\it Case 1.}  $\mu_1>2$.  Choose $\alpha=\frac{6-n}{2}$. Then $\alpha(\alpha+n-2)-\lambda_1=4-\mu_1^2<0$.

Set $u^{+}=u_{V}+A_0u_{V}r^{2}+A_1r^{\frac{6-n}{2}}+A_2r^{\frac{6-n}{2}}\phi_1$, where $A_0$, $A_1$ and $A_2$ are three positive constants to be determined.
Then we have

\begin{align*}
&\Delta u^{+}-\frac{n(n-2)}{4}(u^{+})^{\frac{n+2}{n-2}}\\
=&\Delta u^{+}-\frac{n(n-2)}{4}u_{V}^{\frac{n+2}{n-2}}\bigg[1+A_0r^{2}+\frac{A_1r^{\frac{6-n}{2}}+A_2 r^{\frac{6-n}{2}}\phi_1}{u_{V}}\bigg]^{\frac{n+2}{n-2}}\\
\leq&\Delta u^{+}-\frac{n(n-2)}{4}u_{V}^{\frac{n+2}{n-2}}\big[1+\frac{n+2}{n-2}\big(A_0r^{2}+\frac{A_1r^{\frac{6-n}{2}}+A_2 r^{\frac{6-n}{2}}\phi_1}{u_{V}}\big)\big]\\
=&\Delta (A_1r^{\frac{6-n}{2}}+A_2r^{\frac{6-n}{2}}\phi_1)-\frac{n(n+2)}{4}(r\rho)^{-2}[A_1r^{\frac{6-n}{2}}+A_2r^{\frac{6-n}{2}}\phi_1]\\
&-A_0nu_{V}\rho^{-2}+A_04u_{V}\\
=&r^{-\frac{n-2}{2}}\big[\frac{(6-n)(n+2)}{4}A_1-\frac{n(n+2)}{4}\rho^{-2}A_{1}+(4-\mu_1^2)A_2\phi_1\\
&+A_0\rho^{-\frac{n-2}{2}}(-n\rho^{-2}+4)\big]\\
 \end{align*}
Since $4-\mu_1^2<0$, we can take two constants $k_1, k_2>1$ depending only on $n$ and $\Sigma$ such that
\begin{align*}
k_1\frac{(6-n)(n+2)}{4}-k_1\frac{n(n+2)}{4}\rho^{-2}+k_2(4-\mu_1^2)\phi_1
+\rho^{-\frac{n-2}{2}}(-n\rho^{-2}+4)
\leq  -1-\rho^{-\frac{n+2}{2}}
 \end{align*}
in $\Sigma$. Take $A_0=10(A+B)(C_{L}+1)r_0^{\frac{n-6}{2}}$ for some positive
\begin{equation}\label{r0}r_0\leq \min\{\frac{1}{100(A+B)(C_L+1)},r_1\} \end{equation}to be determined, and $A_1=k_1A_0$, $A_2=k_2A_0$. Then by \eqref{estim-2-B},
\begin{equation}\label{u-mp-bdy1}
u\leq u^+  \quad\text{in }V\cap \partial B_{r_0}(0)
\end{equation}
and
\begin{align}\label{uplus}\begin{split}
&\Delta u^{+}-\frac{n(n-2)}{4}(u^{+})^{\frac{n+2}{n-2}}\\
\leq& -A_0r^{-\frac{n-2}{2}}\big(1+\rho^{-\frac{n+2}{2}}\big) \quad\text{in }V\cap  B_{r_0}(0).\end{split}
 \end{align}

Fix $r_2>0$ such that,
\begin{equation*}
10(A+B)(C_{L}+1)r_2^{\frac{n-2}{2}}[1+\max_{\Sigma}\rho^{\frac{n-2}{2}}(k_1+k_2\max_{\Sigma}\phi_1)]<\frac{1}{1000}.
\end{equation*}
Then, when $r_0\leq \min\{r_1, r_2\}$,
\begin{equation}\label{smallterm}
A_0r^{2}+\frac{A_1r^{\frac{6-n}{2}}+A_2 r^{\frac{6-n}{2}}\phi_1}{u_{V}}<\frac{1}{1000}.
\end{equation}
in $V\bigcap B_{r_0}(0)$.

By Lemma \ref{lemma-phi-estimate}, \eqref{esti-u-A} and the fact $u_{V}= r^{-\frac{n-2}{2}}\rho^{-\frac{n-2}{2}}$, we have
\begin{equation}\label{sec4-1}
A_0u_{V}r^{2}+A_1r^{\frac{6-n}{2}}+A_2r^{\frac{6-n}{2}}\phi_1\leq C_1A_{0} Au_{V}r^2,
\mathrm{\mathrm{\mathrm{\mathrm{}}}}\end{equation}
\begin{align}\label{sec4-2}\begin{split}
\sum_{i=1}^{n}|(A_0u_{V}r^{2}+A_1r^{\frac{6-n}{2}}+A_2r^{\frac{6-n}{2}}\phi_1)_{i}|&\leq  \frac{C_1}{2}A_0r[u_{V}(1+\rho^{-1})+r|\nabla u_{V}(x)|]\\
& \leq C_1A_{0} Au_{V}r(\rho^{-1}+1)
\end{split}\end{align}
and
\begin{align}\label{sec4-3}\begin{split}
&\sum_{i,j=1}^{n}|(A_0u_{V}r^{2}+A_1r^{\frac{6-n}{2}}+A_2r^{\frac{6-n}{2}}\phi_1)_{ij}|\\
\leq &\frac{C_1}{2}A_{0}[u_{V}+r\sum_{i=1}^{n}|\partial_{i} u_{V}(x)|+r^2\sum_{i,j=1}^{n}|\partial_{ij}^2u_{V}(x)|+u_V\rho^{-2}]\\
\leq &C_1A_{0}Au_{V}(\rho^{-2}+1),
\end{split}\end{align}
for some constant $C_1>1$ depending only on $n$ and $\Sigma$.

Fix a $r_3>0$ such that,\begin{equation}\label{sec4-4}10C_1(A+B)(C_{L}+1)r_3^{\frac{n-2}{2}}<\frac{1}{1000}.\end{equation}
Then, when $r_0\leq \min\{\frac{1}{100(A+B)(C_L+1)},r_1, r_2,r_3\}$,
\begin{equation*}
C_1A_{0}r^2<\frac{1}{1000}\quad \quad\quad\text{for }r\in (0,r_0].
\end{equation*}
Combining \eqref{sec4-1}, \eqref{sec4-2} and \eqref{sec4-3}, it holds
\begin{align*}
r^2\sum_{i,j=1}^{n}|\partial_{ij}^2 u^{+}|+r\sum_{i=1}^{n}|\partial_{i} u^{+}|+u^{+}\leq2Au_{V}(\rho^{-2}+1).
 \end{align*}
Then by \eqref{structurecondition},
\begin{align*}
|(L-\Delta)u^{+}|\leq&  C_{L}(r^2\sum_{i,j=1}^{n}|\partial_{ij}^2 u^{+}|+r\sum_{i=1}^{n}|\partial_i u^{+}|+u^{+})\\
\leq& 2C_{L}Au_{V}(\rho^{-2}+1)\\
<& \frac{1}{5}A_0u_{V}(\rho^{-2}+1) \quad\text{in }V\cap  B_{r_0}(0).
 \end{align*}
Hence, by \eqref{uplus},
\begin{align*}
    &Lu^{+}-\frac{n(n-2)}{4}(u^+)^{\frac{n+2}{n-2}}\\
    =&\Delta u^+-\frac{n(n-2)}{4}(u^+)^{\frac{n+2}{n-2}}+(L-\Delta)u^{+}\\
     \leq & -A_0r^{-\frac{n-2}{2}}\big[1+\rho^{-\frac{n+2}{2}}]+| (L-\Delta)u^{+}|\\
    \leq & -A_0r^{-\frac{n-2}{2}}\big[1+\rho^{-\frac{n+2}{2}}]+ \frac{1}{5}A_0 r^{-\frac{n-2}{2}}\rho^{-\frac{n-2}{2}}(\rho^{-2}+1)\\
     < & -\frac{A_0}{2}r^{-\frac{n-2}{2}}(1+\rho^{-\frac{n+2}{2}})\\
    <& 0
\end{align*}
in $V\cap  B_{r_0}(0)$.

 Let $r_4>0$ such that $u_V>2\big(\frac{4C_L}{n(n+2)}\big)^{\frac{n-2}{4}}$ in $V\bigcap B_{r_4}(0)$. Then take $$r_0\leq \min\{\frac{1}{100(A+ B  )(C_L+1)},r_1, r_2,r_3,r_4\}.$$ Arguing as in the proof of Lemma \ref{A},
the maximum principle implies \begin{equation*}u\leq u^{+}=u_{V}\big[1+\big(A_0+A_1\rho^{\frac{n-2}{2}}+A_2\rho^{\frac{n-2}{2}}\phi_1\big)r^2\big] \quad\quad\quad\text{in }V\cap B_{r_0}(0).
\end{equation*}

Set $u^{-}=u_{V}-A_0u_{V}r^{2}-A_1r^{\frac{6-n}{2}}-A_2r^{\frac{6-n}{2}}\phi_1$.
Noting that by \eqref{r0} and \eqref{smallterm}, $\big[\big(A_0+A_1\rho^{\frac{n-2}{2}}+A_2\rho^{\frac{n-2}{2}}\phi_1\big)r^2\big]^2$ is sufficiently small compared to
$\big(A_0+A_1\rho^{\frac{n-2}{2}}+A_2\rho^{\frac{n-2}{2}}\phi_1\big)r^2$.
A similar argument yields
\begin{equation*}u\geq u^{-}=u_{V}\big[1-\big(A_0+A_1\rho^{\frac{n-2}{2}}+A_2\rho^{\frac{n-2}{2}}\phi_1\big)r^2\big] \quad\quad\quad\text{in }V\cap B_{r_0}(0).
\end{equation*}

{\it Case 2.}  $\mu_1=2$.

For any constant $\alpha$, a straightforward calculation yields
\begin{align*}
\Delta (r^{\alpha}\ln r\phi_1)=&r^{\alpha-2}\ln r\big[\alpha(\alpha+n-2)\phi_1+\frac{n(n+2)}{4}\rho^{-2}\phi_1-\lambda_1\phi_1\big]\\
&+(2\alpha+n-2)r^{\alpha-2}\phi_1.
 \end{align*}
Taking $\alpha=\mu_1-\frac{n-2}{2}=\frac{6-n}{2}$, we have
\begin{equation*}
\Delta (\phi_1r^{\frac{6-n}{2}}\ln r)-\frac{n(n+2)}{4}\rho^{-2}\phi_1r^{\frac{2-n}{2}}\ln r=4r^{-\frac{n-2}{2}}\phi_1
\end{equation*}

Set $u^{+}=u_{V}+A_0u_{V}r^{2}+A_1r^{\frac{6-n}{2}}-A_2\phi_1r^{\frac{6-n}{2}}\ln r$, where $A_0$, $A_1$ and $A_2$ are three positive constants to be determined.
Then,

\begin{align*}
&\Delta u^{+}-\frac{n(n-2)}{4}(u^{+})^{\frac{n+2}{n-2}}\\
\leq&\Delta (A_1r^{\frac{6-n}{2}}-A_2\phi_1r^{\frac{6-n}{2}}\ln r)-\frac{n(n+2)}{4}(r\rho)^{-2}[A_1r^{\frac{6-n}{2}}-A_2\phi_1r^{\frac{6-n}{2}}\ln r]\\
&-A_0nu_{V}\rho^{-2}+A_04u_{V}\\
=&r^{-\frac{n-2}{2}}\big[\frac{(6-n)(n+2)}{4}A_1-\frac{n(n+2)}{4}\rho^{-2}A_{1}-4A_2\phi_1\\
&+A_0\rho^{-\frac{n-2}{2}}(-n\rho^{-2}+4)\big]\\
 \end{align*}
First take two constants $k_1, k_2>1$ depending only on $n$ and $\Sigma$ such that
\begin{align*}
k_1\frac{(6-n)(n+2)}{4}-k_1\frac{n(n+2)}{4}\rho^{-2}-4k_2\phi_1
+\rho^{-\frac{n-2}{2}}(-n\rho^{-2}+4)
\leq  -1-\rho^{-\frac{n+2}{2}}
 \end{align*}
As in {\it Case 1}, we can take $A_0=10(A+B)(C_{L}+1)r_0^{\frac{n-6}{2}}$ for some positive $r_0\leq \min\{\frac{1}{100(A+B)(C_L+1)},r_1\}$ small and $A_1=k_1A_0$, $A_2=k_2A_0$. Then the argument following is basically the same as {\it Case 1} and one gets a constant $r_0>0$ small depending only on $n$ and $\Sigma$ such that
 \begin{equation*}u\leq u^{+}=u_{V}\big[1+\big(A_0+A_1\rho^{\frac{n-2}{2}}-A_2\rho^{\frac{n-2}{2}}\phi_1\ln r\big)r^2\big] \quad\quad\quad\text{in }V\cap B_{r_0}(0),
\end{equation*}
and
 \begin{equation*}u\geq u^{-}:=u_{V}\big[1-\big(A_0+A_1\rho^{\frac{n-2}{2}}-A_2\rho^{\frac{n-2}{2}}\phi_1\ln r\big)r^2\big] \quad\quad\quad\text{in }V\cap B_{r_0}(0).
\end{equation*}

{\it Case 3.} $\mu_1<2$.  Then
\begin{equation*}
\Delta (r^{\mu_1-\frac{n-2}{2}}\phi_1)-\frac{n(n+2)}{4}\rho^{-2}r^{\mu_1-\frac{n-2}{2}-2}\phi_1=0.
\end{equation*}

Set $u^{+}=u_{V}+A_0u_{V}r^{2}+A_1r^{\frac{6-n}{2}}+A_2(r^{\mu_1-\frac{n-2}{2}}-r^{\frac{6-n}{2}})\phi_1$, where $A_0$, $A_1$ and $A_2$ are three positive constants to be determined.
Then,

\begin{align*}
&\Delta u^{+}-\frac{n(n-2)}{4}(u^{+})^{\frac{n+2}{n-2}}\\
\leq &\Delta (A_1r^{\frac{6-n}{2}}+A_2r^{\mu_1-\frac{n-2}{2}}\phi_1-A_2r^{\frac{6-n}{2}}\phi_1)\\-&\frac{n(n+2)}{4}(r\rho)^{-2}[A_1r^{\frac{6-n}{2}}+A_2r^{\mu_1-\frac{n-2}{2}}\phi_1-A_2r^{\frac{6-n}{2}}\phi_1]
-A_0nu_{V}\rho^{-2}+A_04u_{V}\\
= &r^{-\frac{n-2}{2}}\big[\frac{(6-n)(n+2)}{4}A_1-\frac{n(n+2)}{4}\rho^{-2}A_{1}-(4-\mu_1^2)A_2\phi_1\\
&+A_0\rho^{-\frac{n-2}{2}}(-n\rho^{-2}+4)\big].\\
 \end{align*}

Noting $4-\mu_1^2>0$, we can take two constants $k_1, k_2>1$ depending only on $n$ and $\Sigma$ such that
\begin{align*}
k_1\frac{(6-n)(n+2)}{4}-k_1\frac{n(n+2)}{4}\rho^{-2}-k_2(4-\mu_1^2)\phi_1
+\rho^{-\frac{n-2}{2}}(-n\rho^{-2}+4)
\leq  -1-\rho^{-\frac{n+2}{2}}
 \end{align*}
in $\Sigma$.
As in {\it Case 1}, we can take $A_0=10(A+B)(C_{L}+1)r_0^{\frac{n-6}{2}}$ for some positive $r_0\leq \min\{\frac{1}{100(A+B)(C_L+1)},r_1\}$ small and $A_1=k_1A_0$, $A_2=k_2A_0$. Then the argument following is basically the same as {\it Case 1} and we have
\begin{equation*}|u-u_{V}|\leq u_{V}\big[A_0r^2+A_1\rho^{\frac{n-2}{2}}r^2+A_2\rho^{\frac{n-2}{2}}\phi_1(r^{\mu_1}-r^{2})\big] \quad\quad\quad\text{in }V\cap B_{r_0}(0).
\end{equation*}
for some constant $r_0>0$ small depending only on $n$ and $\Sigma$.

\end{proof}

\begin{prop}\label{mu-convex cone}
For $n=3,4,5$, if $V\subseteq \mathbb R^n_+$, then $\mu_1>2$.
\end{prop}

\begin{proof}
If $\hat{V}:=\mathbb R^n_+$, then $\Sigma=\mathbb S^{n-1}_+$ and $\rho_{\mathbb S^{n-1}_+}(x)=x_n\leq 1$  on $\mathbb S^{n-1}_+$.
 Then by \eqref{1st-eigenvalue}, $\lambda_1(L_{\mathbb S^{n-1}_+})>\frac{n(n+2)}{4}$. Hence, $\mu_1(L_{\mathbb S^{n-1}_+})>2$.  Then we can draw the conclusion by Proposition \ref{eigen-comparision}.

\end{proof}

\begin{corollary}\label{345-convex cone}
For $n=3,4,5$, if $V$ is convex, then $\mu_1>2$.
\end{corollary}
\begin{proof}
By a rotation, we can assume $V\subseteq \mathbb R^n_+$. We can apply Proposition \ref{mu-convex cone} to conclude
the desired result.

\end{proof}

We are ready to prove Theorem \ref{main reslut2} in this paper.

\begin{proof}[Proof of Theorem \ref{main reslut2}]

Under Assumption \ref{assumption-basic},
\begin{align}\label{rdl}
   g_{ij}=\delta_{ij}+O(|x|^2),\quad g^{ij}=\delta_{ij}+O(|x|^2),
\end{align}
and
\begin{align*}
-L_g u:= \Delta_g u-\frac{n-2}{4(n-1)}S_g u=
 \sum_{i,j=1}^{n}\frac{1}{\sqrt{\det g}}\frac{\partial}{\partial x_i}(\sqrt{\det g}g^{ij}\frac{\partial}{\partial x_j} u)-\frac{n-2}{4(n-1)}S_g u.
\end{align*}

In this coordinate system, the conformal Laplacian $-L_g$ satisfies the structure condition \eqref{structurecondition} in $B_2(0)$
for some positive constant $C_{-L_g }$ depending only on $g$.

Therefore, by taking
\begin{equation*}
\alpha=
\begin{cases}
2
& \text{if }  \mu_1>2,\\
 \text{any number in } (1,2)
& \text{if }   \mu_1=2,\\
\mu_1
& \text{if } \mu_1<2,\\
\end{cases}
\end{equation*}
we have Theorem \ref{main reslut2} by Theorem \ref{thm-cone}.

 In particular, by \eqref{Kmu1-19}, if $n\geq6$, then $\mu_1>2$. Hence, $\alpha$ in \eqref{main-estimate2} can be $2$. For $n=3,4,5$, by Corollary \ref{345-convex cone}, we
 can take $\alpha=2$ in \eqref{main-estimate2} when $V_0$ is convex.

\end{proof}

\section{domains locally bounded by $C^2$ hypersurfaces}\label{sec-domains locally bounded by hypersurfaces}

In this section, we study the asymptotic behaviors of positive solutions of \eqref{eq-MEq}-\eqref{eq-MBoundary} in domains which are locally bounded by $k$
linearly independent $C^2$ hypersurfaces and prove Theorem \ref{main reslut}. The $C^{2}$-diffeomorphism $T$ given by \eqref{C2diffeormorphism1} and the lower bound of $\lambda_1(L_{\Sigma})$ derived in Section \ref{sec-cone} play an important role in the proof for $n=3$.

\begin{theorem}\label{C} Let $\Omega\subseteq\mathbb R ^{n}$ and $\partial\Omega \bigcap B_1(0)$ be bounded by $k$ $C^2$ hypersurfaces $S_1, \cdots, S_k$ with the property that the normal vectors
of $S_1, \cdots, S_k$ at $0$ are linearly independent as in definition \ref{def-Domain-lipschitz}. Let $L$ be the operator given by \eqref{L} and satisfy \eqref{structurecondition}. Suppose $ u \in
C^{\infty}(\Omega\bigcap B_{1}(0))$ is a positive solution of
\begin{align}\label{u-ngeq4}\begin{split}
   L u=&\frac{n(n-2)}{4}u^\frac{n+2}{n-2}\quad in\quad \Omega \cap B_{1}(0),\\
   u=&+\infty\quad on\quad \partial \Omega \cap B_{1}(0),\end{split}
\end{align}
and $u_{V_{0}}$ is the solution of \eqref{eq-MainEq}-\eqref{eq-MainBoundary} in the tangent cone $V_{0}$
of  $\Omega$ at $0$. Then there exists a positive constant $R$, depending only on $n$, $V_0$, $C_{L}$ and the $C^2$-seminorms of $S_1,\cdots, S_k$ in $B_{\frac{1}{2}}(0)$ such that
\begin{align}\label{estimate-ge4}
    \Big|\frac{u(x)}{u_{V_{0}}(Tx)}-1\Big|\leq C|x|^{\alpha}, \quad\quad\quad  \text{in } \Omega\cap  B_{R}(0),
\end{align} where $T$ is the $C^{2}$-diffeomorphism given by \eqref{C2diffeormorphism1}, $\alpha=1$ for $n\geq 4$ and $\alpha=\frac{1}{2}$ for $n=3$, and $C$ is some constant depending only on $n$, $V_0$, $C_{L}$ and the $C^2$-seminorms of $S_1,\cdots, S_k$ in $B_{\frac{1}{2}}(0)$.
\end{theorem}
\begin{proof}
Let $\tilde{u}$ be the positive solution to
\begin{align}\label{u-tilde}\begin{split}
   \Delta\tilde{u}=&\frac{n(n-2)}{4}\tilde{u}^\frac{n+2}{n-2}\quad \text{in}\quad \Omega\cap B_{1}(0),\\
  \tilde{u}=&+\infty\quad \text{on}\quad \partial \big(\Omega\cap B_{1}(0)\big),
\end{split}\end{align}
Then by Lemma \ref{A}, there exists a positive constant $r_1\leq\frac{1}{2}$, such that
\begin{align}\label{u-tildeu}
    |u-\tilde{u}|\leq Cu|x|^{\alpha} \quad\quad\quad  \text{in } \Omega\cap  B_{r_1}(0),
\end{align}
where $C$ is some constant depending only on $n$, $V$, $C_{L}$ and the $C^2$-seminorms of $S_1,\cdots,S_k$ in $B_{\frac{1}{2}}(0)$.

Next, we consider the behavior of $\tilde{u}(x).$
According to Theorem 1.1 in \cite{hanshen2}, there exists a positive constant $r_2\leq\frac{1}{2}$, such that
\begin{align}\label{u-uv}
    |\tilde{u}-u_{V_0}(Tx)|\leq C u_{V_0}(Tx)|x| \quad\quad\quad  \text{in }  \Omega \cap B_{r_2}(0),
\end{align}
where $C$ is some constant depending only on $n$, $V_0$, $C_{L}$ and the $C^2$-seminorms of $S_1,\cdots, S_k$ in $B_{\frac{1}{2}}(0)$.

Take $R=\min\{r_1,r_2\}$, we have \eqref{estimate-ge4}.

\end{proof}

\smallskip
We now improve the estimate \eqref{estimate-ge4} for $n=3$ with the aid of $T$ and $\lambda_1$.
\begin{theorem}\label{E} Let $\Omega\subseteq\mathbb R ^{3}$ and $\partial\Omega \bigcap B_1(0)$ be bounded by $k$ $C^2$ hypersurfaces $S_1, \cdots, S_k$ near $0$  with the property that the normal vectors
of $S_1, \cdots, S_k$ at $0$ are linearly independent as in definition \ref{def-Domain-lipschitz}.
 Let $L$ be given by \eqref{L} and satisfy \eqref{structurecondition}. Suppose that $ u \in
C^{\infty}(\Omega\bigcap B_{1}(0))$ is a positive solution of
\begin{align}\label{u-n-3}\begin{split}
   L u=&\frac{3}{4}u^5\quad   \text{in }\quad \Omega \cap B_{1}(0),\\
   u=&+\infty\quad   \text{on }\quad \partial \Omega \cap B_{1}(0),\end{split}
\end{align}
and $u_{V_{0}}$ is the solution of \eqref{eq-MainEq}-\eqref{eq-MainBoundary} in the tangent cone $V_{0}$
of $\Omega$ at $0$.  Then, there exists a positive constant $R$ depending only on $V_0$, $C_{L}$ and the $C^2$-seminorms of $S_1,\cdots, S_k$ in $B_{\frac{1}{2}}(0)$ such that
\begin{align}\label{estimate-n-3}
    \Big|\frac{u(x)}{u_{V_{0}}(Tx)}-1\Big|\leq C|x|, \quad\quad\quad  \text{in } \Omega\cap  B_{R}(0),
\end{align} where $T$ is the $C^{2}$-diffeomorphism given by \eqref{C2diffeormorphism1}, and $C$ is some constant depending only on $V_0$, $C_{L}$ and the $C^2$-seminorms of $S_1,\cdots, S_k$ in $B_{\frac{1}{2}}(0)$.
\end{theorem}
\begin{proof}

Set $\Sigma=V_0\cap \mathbb{S}^{2}$. For $n=3$, $u_{V_{0}}=|x|^{-\frac{1}{2}}g=r^{-\frac{1}{2}}\rho^{-\frac{1}{2}}$, where $r=|x|$ and $\rho=g^{-2}$. By \eqref{eq-estimate-rho-upper-lower} and (Lemma 3.2, Remark 3.3, Lemma 3.4) in \cite{hanshen2}, we have
 \begin{equation}\label{u-growth-3}A^{-1}\leq (d(x,\partial V_0))^\frac{n-2}{2}u_{V_0}\leq 2^\frac{n-2}{2},\end{equation}
and
\begin{equation}
u_{V_0}+r\rho\sum_{i=1}^{n}|\partial_i u_{V_0}(x)|+(r\rho)^2\sum_{i,j=1}^{n}|\partial_{ij}^2u_{V_0}(x)|\le A u_{V_0} \quad \quad\text{in }V_0,\end{equation}
where $A>1$ is a constant depending only on $\Sigma$.

By Theorem \ref{C}, there exists a constant $r_1\in(0,\frac{1}{2})$ depending on
$V_0$, $C_{L}$ and the $C^2$-seminorms of $S_1,\cdots, S_k$ in $B_{\frac{1}{2}}(0)$, such that
\begin{align}\label{u-uV-3-1/2}
    \Big|\frac{u(x)}{u_{V_{0}}(Tx)}-1\Big|\leq B|x|^{\frac{1}{2}},
\end{align}
In the following, we require $r_1$ small such that $Br_1^{\frac{1}{2}}<\frac{1}{1000}$.

Set $\overline{u}=u_{V_0}+A_0u_{V_0}r+A_1r^{\frac{1}{2}}+A_2r^{\frac{1}{2}}\phi_1$, where $A_0$, $A_1$ and $A_2$ are three positive constants to be determined.

\begin{align*}
&\Delta \overline{u}-\frac{3}{4}\overline{u}^{5}\\
<&\Delta u_{V_0} +A_0\big(\frac{3}{4}u_{V_0}^{5}r+u_{V_0}r^{-1}\big)+r^{-\frac{3}{2}}\big[\frac{3}{4}(A_1+A_2\phi_1)+A_2(\frac{15}{4}\rho^{-2}\phi_1-\lambda_1\phi_1)\big]\\
&-\frac{3}{4}u_{V_0}^{5}\big[1+5\big(A_0r+\frac{A_1r^{\frac{1}{2}}+A_2 r^{\frac{1}{2}}\phi_1}{u_{V_0}}\big)\big]\\
=&r^{-\frac{3}{2}}\big[\frac{3}{4}A_1-\frac{15}{4}\rho^{-2}A_{1}+(\frac{3}{4}-\lambda_1)A_2\phi_1+A_0\rho^{-\frac{1}{2}}(-3\rho^{-2}+1)\big].
\end{align*}
By Proposition \ref{eigen-esti}, $\frac{3}{4}-\lambda_1<0$. Hence, we can take two constants $k_1,k_2\geq1$ such that
\begin{equation*}\frac{3}{4}k_1-\frac{15}{4}\rho^{-2}k_{1}+(\frac{3}{4}-\lambda_1)k_2\phi_1+\rho^{-\frac{1}{2}}(-3\rho^{-2}+1)\leq -\rho^{-\frac{5}{2}}-1 \quad\quad\quad\text{in } \Sigma.
\end{equation*}
Then taking $A_0=10(A+B)(C_{L}+1)r_0^{-\frac{1}{2}}$ for some positive $r_0\leq \min\{\frac{1}{100(A+B)(C_L+1)},r_1\}$ to be determined, and $A_1=k_1A_0$, $A_2=k_2A_0$, we have
Hence
\begin{equation*}
\Delta \overline{u}-\frac{n(n-2)}{4}\overline{u}^{\frac{n+2}{n-2}}\leq -A_0r^{-\frac{3}{2}}(\rho^{-\frac{5}{2}}+1) \quad\text{in }V_0\cap  B_{r_0}(0).
\end{equation*}
As in the proof of Theorem \ref{thm-cone}, we can take $r_2>0$ small such that when $r_0\leq \min\{\frac{1}{100(A+B)(C_L+1)},r_1,r_2\}$, we have
\begin{align}\label{est-u-cone2}
r^2\sum_{i,j=1}^{3}|\partial_{ij}^2  \overline{u}|+r\sum_{i=1}^{3}|\partial_{i}  \overline{u}|+ \overline{u}\leq2Au_{V_0}(\rho^{-2}+1),\quad\quad\quad\text{in } V_0\cap B_{r_0}(0).
 \end{align}
and therefore
\begin{equation}\label{est-u-cone}
L \overline{u}-\frac{3}{4}\overline{u}^{5}\leq -\frac{A_0}{2}r^{-\frac{3}{2}}(\rho^{-\frac{5}{2}}+1)\quad\quad\text{in } V_0\cap B_{r_0}(0).
\end{equation}

Let $T$ be the $C^{2}$-diffeomorphism defined by \eqref{C2diffeormorphism1} in $B_{r_3}(0)$ for some constant
$r_3\in(0,\frac{1}{2})$ depending only on $V_0$ and the $C^2$-norms of $S_1,\cdots, S_k$ in $B_{\frac{1}{2}}(0)$.
Set $u^+(x)=\overline{u}\circ T(x)=\overline{u}(Tx)$.
By \eqref{diffeo-0}, for small $r_4,$
\begin{equation}\label{Tx-x}|Tx-x|\leq \frac{1}{100}|x|,\end{equation}
and
\begin{equation}\label{dist}|\frac{d(x,\partial\Omega)}{d(Tx, \partial V_0)}-1|\leq \frac{1}{100}.\end{equation}
Combining with \eqref{u-uV-3-1/2}, we have
$u<u^+$ on $\Omega\bigcap\partial B_{r_0}(0)$ if $r_0\leq\min\{\frac{1}{100(A+B)(C_L+1)},r_1,r_2,r_3,r_4\}$.

 Now we estimate $Lu^+.$ By \eqref{nablaf} and \eqref{hessf},

\begin{align*}
\big(L u^+(x)-\frac{3}{4}(u^+)^{5}\big)|_{x}=L \overline{u}|_{Tx}+E-\frac{3}{4}\overline{u}^{5}|_{Tx},
\end{align*} where $E$ is the error term. By \eqref{nablaf} and \eqref{hessf}, we have
\begin{align}\label{ee}
    |E|\leq C_{T}(|\nabla u^+||_{Tx}
+|x||\nabla^2(u^+)||_{Tx}).
\end{align}
where $C_T$ is some constant depending only on $V_0$ and the $C^2$-seminorms of $S_1,\cdots,S_k$ in $B_{\frac{1}{2}}(0)$.

By \eqref{est-u-cone},
\begin{align*}
  \big[Lu^+-\frac{3}{4} (u^+)^{5}\big]|x\leq -\frac{A_0}{2}\bigg(r^{-\frac{3}{2}}(\rho^{-\frac{5}{2}}+1)\bigg) \circ T +|E|.
\end{align*}

By \eqref{difxbarx},
$|Tx-x|\leq C_T|x|^2$. Combing with \eqref{u-growth-3}, \eqref{est-u-cone2}, \eqref{Tx-x} and \eqref{dist}, one gets
\begin{align*}
    |E|\leq  AC_1\bigg( r^{-1}u_{V_0}(\rho^{-2}+1)\bigg)\circ T ,
\end{align*} where $C_1$ is some constant depending on $V_0$, $C_{L}$ and the $C^2$-seminorms of $S_1,\cdots, S_k$ in $B_{\frac{1}{2}}(0)$  but independent of $r_0$ small.
Hence, if we further require $r_0$ to be sufficiently small, by the definition of $A_0$,
we have
\begin{align*}
L u^+-\frac{3}{4} (u^+)^{5}  \leq -\frac{A_0}{4}\bigg(r^{-\frac{3}{2}}(\rho^{-\frac{5}{2}}+1)\bigg) \circ T<0
\end{align*} in $\Omega\bigcap B_{r_0}(0)$. Then by the maximum principle, similarly as in the proof of Lemma \ref{A}, $u\leq u^+.$
For the lower bound, we can take $$u^-=\big(u_{V_0}-A_0u_{V_0}r-A_1r^{\frac{1}{2}}-A_2r^{\frac{1}{2}}\phi_1\big)\circ T,$$ as a sub-solution and proceed similarly. Then we have $u^-\leq u\leq u^+$ in $\Omega\cap B_{r_0}(0)$.
\end{proof}

\smallskip
We are ready to prove Theorem \ref{main reslut} in this paper.

\begin{proof}[Proof of Theorem \ref{main reslut}]

Under Assumption \ref{assumption-basic},
\begin{align}\label{rdl}
   g_{ij}=\delta_{ij}+O(|x|^2),\quad g^{ij}=\delta_{ij}+O(|x|^2),
\end{align}
and
\begin{align*}
-L_g u:= \Delta_g u-\frac{n-2}{4(n-1)}S_g u=
 \sum_{i,j=1}^{n}\frac{1}{\sqrt{\det g}}\frac{\partial}{\partial x_i}(\sqrt{\det g}g^{ij}\frac{\partial}{\partial x_j} u)-\frac{n-2}{4(n-1)}S_g u.
\end{align*}

In this coordinate system, the conformal Laplacian $-L_g $ satisfies the structure condition \eqref{structurecondition} in $B_2(0)$ for some positive constant $C_{-L_g }$ depending only on $g$.

Then, by Theorem \ref{C} and Theorem \ref{E}, we have
\begin{align}\label{utildeu}
\Big|\frac{u(x)}{u_{V_{0}}(Tx)}-1\Big|\leq C|x|.
\end{align} Next, we consider the behavior of $u_{V_{0}}(Tx).$

By \eqref{eq-Solution-Cone-d-coordinates-nd2} and \eqref{C2diffeormorphism1}, we have $$u_{V_{0}}(Tx)=f_{V_0}(d_1,\cdots,d_k),$$
where $d_i$ is the signed distance from $x$ to $S_i$ under the Euclidean metric in the normal coordinate system under Assumption \ref{assumption-basic}. Note that under this normal coordinate system, we have $d_{i}=d_{g,i}(1+O(d_{g}(x,0)))$, where $d_{g,i}$  is the signed distance from $x$ to $S_i$ under the  metric $g$, $i=1,...,k$. Then by the anisotropic gradient estimates in Lemma 3.5 in \cite{hanshen2}, which is also applicable to $k<n$ case as remarked after the proof of Lemma 3.5, we have
\begin{align*}
   | f_{V_0}(d_1,\cdots,d_k)-f_{V_0}(d_{g,1}\cdots,d_{g,k})|\leq C f_{V_0}(d_{g,1}\cdots,d_{g,k}) d_{g}(x,0),
 \end{align*}
 where $C$ is some constant depending only on $n$, $V_0$ and $g$.

 Let $T_g$ be the $C^2$ diffeormorphism given by \eqref{C2diffeormorphism}.
 Finally, by the definition of $T_g$ and the relation \eqref{utildeu}, we have
 \begin{align*}
    u(x)=  f_{V_0}(d_{g,1}\cdots,d_{g,k})\big(1+O(d_{g}(x,0))\big).
 \end{align*}

Therefore, \begin{align}
 \Big|\frac{u(x)}{u_{V_{0}}(T_gx)}-1\Big|\leq Cd_{g}(x,0),
\end{align}
where $C$ is some constant depending only on $n$, $V_0$, $g$ and the $C^2$-seminorms of $S_1,\cdots, S_k$ in $B_{\frac{1}{2}}(0)$.

\end{proof}

\end{document}